\documentclass[11pt]{article}
\usepackage{mathtools}
\usepackage{amsmath,amsfonts,amssymb,amsthm}
\usepackage{xypic}
\usepackage{latexsym, xspace, enumerate}
\usepackage[mathscr]{eucal}
\usepackage{hyperref}
\usepackage{amscd}
\usepackage{comment}
\newtheorem{theorem}{Theorem}[section]

\newtheorem{proposition}[theorem]{Proposition}

\newtheorem{lemma}[theorem]{Lemma}

\newtheorem{corollary}[theorem]{Corollary}

\theoremstyle{definition}
\newtheorem{definition}[theorem]{Definition}

\newtheorem{remark}[theorem]{Remark}
\newtheorem{notation}[theorem]{Notation}

\newcommand{\V}{\mathbb{V}}

\newcommand{\Z}{\mathbb{Z}}

\newcommand{\R}{\mathbb{R}}
\newcommand{\N}{\mathbb{N}}

\numberwithin{equation}{section}

\date{\today}

\begin{document}

\title{The spectrum problem for $\ell$-groups and for MV-algebras: a categorical approach}

\author{Giuseppina Gerarda Barbieri, Antonio Di Nola, Giacomo Lenzi}
\maketitle

\begin{abstract} As a main result, we characterize prime spectra of abelian lattice ordered groups. Further we introduce some categories based on spectral spaces, lattices and Priestley spaces, and we relate these categories with each other and with the category of presented MV-algebras, by means of functors. We turn to lattices and offer a simple characterization of 1) maps whose Stone dual preserves closed sets, and 2)  closed epimorphisms between distributive lattices as well as their Stone duals. We have a characterization of the variety generated by the Chang MV-algebra and we study this variety. Next we generalize the results to every variety generated by a Komori chain. Finally we discuss homogeneous polynomials in MV-algebras. 
\end{abstract}


\section{Introduction} 

The spectrum problem for MV-algebras is the problem of characterizing 
the prime ideal spectrum of MV-algebras in pure topological terms. 
This problem appears in the list of open problems of  \cite{M1}.  
Note that by the Mundici $\Gamma$ functor of \cite{M86}, between the categories of MV-algebras and 
abelian lattice ordered groups ($\ell$-groups)  with strong unit, 
the problem is equivalent 
to studying prime spectra of the latter groups. Another functor relevant for the problem is the functor 
$\Delta$ of \cite{DL1} between $\ell$-groups and perfect MV-algebras. Let us recall some literature related to the problem. 

The seminal work on this topic is the paper by  Stone \cite{S}, 
where he proved that the category of Boolean algebras and their homomorphisms is equivalent to the opposite of the category of compact Hausdorff zero-dimensional spaces (known as Stonean spaces or Boolean spaces) and the continuous maps between them. One year later  in \cite{S1} Stone himself
obtained a duality between bounded distributive lattices and the category of spectral spaces and so-called spectral maps.  A further duality, known as Priestley duality, concerning bounded distributive lattices and Priestley Spaces, was established by H. Priestley (see \cite{P70}).

Later,  on this  research direction, in \cite{H}  Hochster 
characterized the prime spectra of commutative rings with unity, 
and proved that they coincide with spectral spaces. 

In \cite{DFT} spectra of groups are  studied, with a suitable definition of prime subgroup and the prime spectrum of a group. 

In \cite[Theorem 3.12]{DP} we have a categorical duality between the category of MV-algebras and a category of sheaves of MV-algebras, where the base space is the prime spectrum of an MV-algebra, and whose stalks are MV-chains. In this sheaf representation, stalks are well understood, but the base space is given by spectra of MV-algebras, a kind of space which was not very well understood at the time (now \cite{DL} clarifies better the situation). 

Likewise, in \cite[Theorem 2.22]{FG95} we have a categorical equivalence between the category of MV-algebras and a category of sheaves of MV-algebras, where the base space is the maximal spectrum of an MV-algebra (that is, a compact Hausdorff space) and whose stalks are local MV-algebras. In this representation, conversely with \cite{DP}, the base space is understood, but the stalks range in a wide class such as local MV-algebras. 

We note that both \cite{DP} and \cite{FG95} give possible realizations of the intuitive idea of MV-space, analogous to the idea of ringed space in algebraic geometry. The fact that the analogous of ringed space for MV-algebras does not exist in literature, is probably due to the fact that the notion of localization is, in general, difficult to realize in MV-algebras.

In \cite{FGM} we find a duality between a generalization of MV-algebras and Priestley spaces with an appropriate extra structure. Note that the Priestley space of a distributive lattice can be realized in various ways,  for instance we can take the set of the prime ideals of the lattice, equipped with the inclusion order and with the patch topology (not with the Zariski topology).

The papers \cite{W19} and \cite{W21} contain a deep study of the spectrum problem from a logical point of view. In  \cite{W19} it is shown that second countable spectral spaces which are prime spectra of MV-algebras are characterized by a simple first order formula, whereas in \cite{W21} is is shown that no infinitary first order logic can express the prime spectra of MV-algebras. Other information is contained in \cite{W22}.

The paper \cite{DL} is a kind of predecessor of the present paper.  
  In \cite{DL}  there is a characterization of those topological spaces $X$ which are homeomorphic to the prime spectrum of an MV-algebra, in terms of the lattice $\overset{\circ}{K}(X)$ of the compact open sets of $X$. Apart from 
spectrality of $X$, which is a well known necessary condition, $\overset{\circ}{K}(X)$ must be a closed epimorphic  image of the lattice of what are called ``cylinder rational polyhedra'' in $[0,1]^Y$ for some set $Y$, possibly infinite. Closedness here is a technical  condition on semilattice homomorphisms (the original definition e.g. in \cite{W19}  is slightly involved, but we give a simple characterization of it in Theorem \ref{thm:closed}). A cylinder rational polyhedron here means the zeroset of a McNaughton function and is a natural generalization of rational polyhedra in possibly infinite dimension. 

Another predecessor of this paper is \cite{MS}, where a duality is given in the particular case when the MV-algebras are semisimple. The topological spaces involved in the duality of \cite{MS} are exactly the compact Hausdorff spaces. 

\subsection{Contribution of this paper}

The main goal of this paper is twofold. First, we wish to 
 understand prime spectra of significant subclasses of MV-algebras or $\ell$-groups, using also the characterization of spectra of arbitrary MV-algebras 
given in \cite{DL}. Second, we propose seemingly new functors involving the class of all MV-algebras, partially building on classical dualities. 

We list the main results of the paper.

We are interested both in MV-algebras and $\ell$-groups, since they are deeply related with each other and they have each its own communities of scholars. For $\ell$-groups, via the particular case of Riesz spaces, we  have for instance the community of functional analysis, and for MV-algebras, we have for instance the community of quantum mechanics. The most significant result of this paper from the point of view of $\ell$-groups is the following.

Let $T$ be a cylinder fan based topological space,  which in fact is a spectral space such that for some cardinal
 $k$, possibly infinite, the lattice of compact open sets of $T$ is a lattice anti isomorphic to the lattice of cylinder fans in $\Delta({\mathbb R})^k$  (see Definition \ref{def:cylfan}).

\begin{theorem}\label{thm:lspec} 
A topological space $X$ is the spectrum of an abelian $\ell$-group $G$
if and only if there is a closed subset $Y$ of a cylinder fan based space $T$, 
such that $X=Y\setminus \{M\}$, where $M$ is the only closed point in $Y$.\end{theorem} 

 This result will be proved in Section \ref{group}  using the equivalence $\Delta$ of \cite{DL1} between $\ell$-groups and perfect MV-algebras. 
Recall that an MV-algebra is perfect when it is generated by its infinitesimals. Note that the theorem does not suppose that $G$ has a strong unit, whereas many results on $\ell$-groups require a strong unit. 

 As a further categorical result on $\ell$-groups, we mention that in \cite{A} Abbadini, Marra and Spada  prove that  the category of metrically complete
$\ell$-groups is dually equivalent to the category of a-normal spaces and a-maps between them, they call  an arithmetic space (a-space, for short)   a topological space equipped with a function,
called the denominator map, that ranges in the set of natural numbers, they call a-normal spaces  compact Hausdorff arithmetic spaces satisfying a separation axiom.

In theorem \ref{thm:kospec} we prove that, if $A=Free(k)$  is a free MV-algebra over $k$ generators, then 
 there is a lattice isomorphism between  the lattice $\overset{\circ}{K}(Spec(A))$ of  the  compact open subsets  of $Spec(A)$ and  the lattice of the cylinder rational polyhedra of $[0,1]^k$.

In  Theorem \ref{thm:fond} we give a characterization of the prime spectra of free MV-algebras, and in Theorem \ref{thm:max} we give a partially analogous characterization of maximal spectra of free MV-algebras (that is, hypercubes).

In Theorem \ref{thm:sic} we establish a functor from the category of presented MV-algebras to a category SIC of ``spectra in context'', which are triples $(X,C,m)$ where $X$ is the spectrum of a free MV-algebra, $C$ is another spectral space, and $m$ is a suitable function from $C$ to $X$.  

In Theorem \ref{thm:siclic}  we establish a duality, based on Stone duality,  between SIC and a category LIC (lattices in context) similar to SIC but involving pairs of lattices and lattice maps. Similarly, in Theorem \ref{thm:Pries} we have a duality, based on Priestley duality, between LIC and another category PIC (Priestley in context) involving pairs of Priestley spaces and Priestley maps. 

In Theorem \ref{thm:closed} we characterize closed surjective maps between bounded distributive lattices (in the sense of \cite{DL} and \cite{W19}).

In Theorem \ref{theor} we characterize the algebras in the variety $\V(C)$ generated by the Chang algebra, in terms of their ideals, and in
 theorem \ref{thm:local} we do the same for local MV-algebras. 

In Theorem \ref{thm:fan} we treat spectra of free MV-algebras in $\V(C)$, and in theorem \ref{thm:mfan} we do the same for $\V(K_m)$, where $K_m$ is the $m$-th Komori algebra.  

\subsection{Outline of the paper}

We begin with the preliminary section \ref{sect:prelim}. In section \ref{Spectra} we recall basic properties of spectra of MV-algebras. In Section 
\ref{Spectrum} we characterize spectra of free MV-algebras and in Section \ref{MV} we give some functors. In Section \ref{sec:lathom} we characterize closed epimorphisms between distributive lattices. 
In section \ref{perfect} we recall perfect MV-algebras and perfect ideals. In Section \ref{closed} we study $\V(C)$-algebras and in 
Section \ref{VC} we study spectra of free MV-algebras in $\V(C)$. In section \ref{komori} we study  the more general case of the Komori varieties $\V(K_m)$. 
In section  \ref{group} we consider  Spectra of  $\ell$-groups.   In section  \ref{geo}  we describe  zerosets   of functions from $\Delta(\mathbb R)^n$ to $\Delta(\mathbb R)$ which can be interpreted as the analogous of McNaughton functions  in $\V(C)$. In section \ref{sec:hom} we perform a study of homogeneous McNaughton functions. In  section \ref{mod} we extend Mundici's ideas about mod/theor connection. In Section \ref{conclusio} we draw some conclusions. 

\section{Preliminaries}\label{sect:prelim}

Recall that an MV-algebra (see e.g. \cite{CDM}) is a structure $(A, \oplus, 0, \neg )$ such that $(A, \oplus, 0)$ is a monoid (necessarily commutative), $\neg \neg x = x$, $x\oplus\neg 0 = \neg 0$, $\neg (\neg x \oplus y) \oplus y = \neg (\neg y \oplus x) \oplus x$.

Every MV-algebra is a lattice where $x\vee y = \neg (\neg x\oplus y)\oplus y$ and $x\wedge y = \neg (\neg x \vee \neg y)$.

In this order $0$ is the minimum and $1=\neg 0$ is the maximum. 

An ideal of an MV-algebra $A$ is a nonempty set $I \subseteq A$ closed under $\oplus$ and such that if $x\in I$ and $y\leq x$, then $y\in I$.

An ideal $P$ of $A$ is prime if $P\not=A$ and $x\wedge y\in P$ implies $x\in P$ or $y\in P$. 

The prime spectrum of an MV-algebra $A$ is the set $Spec(A)$ of the prime ideals of $A$. $Spec(A)$ has the Zariski topology, generated by the open sets $O(a)=\{P\in Spec(A)|a\notin P\}$, where $a\in A$. The complement of $O(a)$ is  $V(a)=\{P\in Spec(A)|a\in P\}$, and more generally for a set $I\subseteq A$ we let $V(I)=\{P\in Spec(A)|I\subseteq  P\}$. 

If $A$ is an MV-algebra and $S\subseteq A$, we denote by $ideal(S)$ the ideal generated by $S$. 

Since axioms for MV-algebras are equational, MV-algebras form a variety, and we have free objects. The free MV algebra over $k$ elements (with $k$ possibly infinite cardinal) $Free(k)$ is given by the McNaughton functions from $[0,1]^k$ to $[0,1]$, that is, the continuous piecewise linear functions with integer coefficients. We denote by $Z(f)$ the zeroset of the function $f$. 

As usual in universal algebra, every MV-algebra is an epimorphic image of a free MV-algebra. Since kernels of homomorphisms of MV-algebras are ideals, every MV-algebra is a quotient of a free MV-algebra modulo an ideal. 

Another key notion of this paper is that of $\ell$-group, or lattice ordered abelian group. That is, we have an abelian group $(G+,0,-)$ together with a lattice order $\leq $ such that $x\leq y$ implies $x+z\leq y+z$. $\ell$-groups are regarded as algebraic structures over the signature $(+,0,-,\wedge,\vee)$. An $\ell$-polynomial is a polynomial written in this signature. $\ell$-groups are also a variety. 

A strong unit of an $\ell$-group $G$ is an element $u$ of $G$ such that for every $x\in G$ there is $n$ with $x\leq nu$. 

We  recall the Mundici functor $\Gamma$, the functor $\Delta$ and  Komori varieties of MV-algebras.

$\Gamma$ is a functor from the category of $\ell$-groups with strong unit to the category of MV-algebras. 

Let $(G, u)$ be an $\ell$-group equipped with a fixed strong unit $u$. Then  the
MV-algebra $\Gamma(G, u) = ([0, u], \oplus,\neg, 0)$ is
defined as follows:
\begin{itemize}
\item $[0, u] = \{a \in G : 0 \leq a \leq  u\}$

\item $a\oplus  b = (a + b) \wedge  u$

\item $\neg a = u - a$

\item $0$  is the additive identity of $G$.
\end{itemize}

For morphisms, given $f:(G,u)\to (H,v)$ a morphism of $\ell$-groups with strong unit, $f$ restricts to a  MV algebra morphism $\Gamma(f):\Gamma(G,u)\to\Gamma(H,v)$. 

It results that $\Gamma$ is an equivalence between the category of $\ell$-groups with strong unit and MV-algebras. 

Now if $G$ is an $\ell$-group, we let 

$$\Delta(G)=\Gamma(\Z\ lex\ G,(1,0)),$$ 

where lex is the lexicographic product of groups. Given an $\ell$ group homomorphism $f:G\to H$, we have a group homomorphism $i:Z\ lex\ G\to Z\ lex\ H$ given by $i(z,g)=(z,f(g))$ and note $i(1,0)=(1,0)$, and $(1,0)$ is a strong unit in both  $\ell$-groups $\Z\ lex\ G$ and $\Z\ lex\ H$, so that we can apply the $\Gamma$ functor to $(\Z\ lex\ G,(1,0))$.

Also $\Delta$ is an equivalence, between the category of $\ell$-groups and the category of perfect MV-algebras, which are those MV-algebras which are generated by their infinitesimals.  

 Let $m\in\mathbb N$. We denote by ${\L}_m= \Gamma(\mathbb Z, m)$ the unique $m+1$ element {\L}ukasiewicz finite chain, and 
 $K_m=\Gamma(\mathbb Z\ lex\, \mathbb Z, (m,0))$ the $m$-th Komori chain.

We denote the variety of all MV-algebras by $MV$. If $\emptyset\not= X\subseteq MV$, then $\V(X)$
is the subvariety of $MV$ generated by $X$. 

Komori in \cite{K} proved that every
proper subvariety $\V$ of  the variety of MV-algebras  is of the form
$$(*) \quad \V = \V({\L}_{m_1},\dots, {\L}_{m_r}
, K_{t_1},\dots,  K_{t_s})$$
for some finite sets $I = \{m_1, \dots , m_r\}$ and $J = \{t_1, \dots , t_s\}$, not both empty.

Recall that a topological space $X$ is called generalized spectral if it is
sober (i.e., every irreducible closed set is the closure of a unique singleton) and
the collection of all compact open subsets of X forms a basis of the topology
of X, closed under intersections of any two members. If, in addition, X is
compact, we say that it is spectral.
 
\section{Spectra of MV-algebras}\label{Spectra}

Belluce  proved that the spectrum of an MV-algebra is spectral  (see \cite{B,BDS})-another proof of this fact is also contained in \cite[Corollary 8.9]{DP}, nevertheless  for  the sake of completeness we here give a complete proof of this classical result for free MV-algebras, see Theorem \ref{thm:kospec} (for arbitrary MV-algebras it will be a consequence).

\begin{lemma}\label{lemm}  If $I_\alpha$ is any collection of ideals, then
$\cap_\alpha V (I_\alpha)=V(\oplus_\alpha  I_\alpha)$\end{lemma}

\begin{proof} Observe that
$\oplus_\alpha  I_\alpha$ is the smallest ideal containing every $I_\alpha$. 

If $P\in V (\oplus I_\alpha)$,
since   $I_\alpha\subseteq  \oplus_\alpha  I_\alpha \subseteq P$, we can conclude that 
$P\in\cap V (I_\alpha)$.

To prove the other direction, suppose $P\in\cap V (I_\alpha)$. 

Hence  for any $\alpha$ the ideal $I_\alpha$ is contained in $P$. Then 
$I_\alpha\subseteq \oplus I_\alpha\subseteq P$. Thus,
$P\in V(\oplus I_\alpha)$.\end{proof} 

\begin{proposition}\label{pr}  Let $A$ be an MV-algebra. Then $Spec(A)$ is sober.\end{proposition} 
\begin{proof} Let  $F = V (I)$   be an irreducible  closed subset of $Spec(A)$. 

 We may suppose that $I$ is the intersection of  a family $Y$ of prime ideals. If  $|Y | > 1$, 
we may write $Y = Y'\cup  Y''$  where $Y'\not=Y$ and $Y''\not=Y$
 Then  $F' = V (\cap_{P\in Y'}  P)$ and $F'' = V (\cap_{P\in Y''}  P)$ are proper closed subsets of $F$ and $F=F'\cup F''$, since $I=(\cap_{P\in Y'}  P)\cap (\cap_{P\in Y''} P)$ and this is a
contradiction, since $F$ is irreducible. So $|Y|=1$ and $F=V(P)$.\end{proof} 

 Let $A$ be an MV-algebra and  $f\in A$.  Recall  $$O(f)=\{P\;\text{prime ideal of} \; A \;\text{s.t }\; f\notin P\}$$ 

\begin{lemma}\label{propo} Let $f,\,g\in A$, then $O({f\oplus g})= O(f)\cup O(g)$\end{lemma}
\begin{proof} Since ideals are closed downwards, we have $O(f),O(g)\subseteq O(f\oplus g)$. For the converse inclusion, since ideals are closed under sum, if $P$ is a prime ideal with $f \oplus g \notin P$, then  $f \in  P$ and  $g\in P$ cannot occur together.
\end{proof} 

\begin{proposition}\label{pro} Let $f\in A$, then  $O(f)$ is compact.
\end{proposition} 
\begin{proof}  Let  $$O(f)\subseteq \cup O{(f_i) }$$ be an open cover of  $O(f)$ . Then  $$V (f) =
A \setminus O(f) =A \setminus  \cup O{(f_i) } =\cap V(f_i)= V (I)$$ where $I$ is the ideal generated by  $f_i$. 
Hence $f=\oplus _{n=1}^k f_n$ and  by Lemma \ref{propo}
$$O(f)\subseteq  \cup_{n=1}^k O({f_n})$$
\end{proof} 

Conversely we have: 

\begin{proposition} If an open $O(I)$ is compact, then $O(I)=O(f)$ for some $f$. 
\end{proposition} 

\begin{proof} Suppose $O(I)$ is compact. Then we have the covering $O(I)=\bigcup_{i\in I}O(i)$ and we can extract a finite subset $F$ of $I$ such that $O(I)=\bigcup_{g\in F}O(g)$. Now take $f=\oplus_{g\in F}g$.   
\end{proof} 

\begin{proposition}\label{prof}  Let  $A=Free(k)$ be  a free MV-algebra over $k$ generators. Let $f\in A$ and $Z(f)$ be the zeroset of $f$ in $[0,1]^k$.
Let $\phi$ be the map from the lattice of compact open sets of $Spec(A)$ to the lattice of zerosets of McNaughton functions in $Free(k)$, such that 
$\phi:O(f)\mapsto Z(f)$. Then $\phi$ is well-defined, i.e. it does not depend on the choice of $f$, and it is a lattice isomorphism with respect the inclusion.\end{proposition} 
\begin{proof} Let $f,\,g$ be two McNaughton functions. It is known that the principal  ideal of $f$ is the intersection of all prime ideals containing $f$. So, $O(f)$ is included in $O(g)$ if and only if  $f$ is in the principal ideal of $g$. Moreover, $Z(f)$ is included in $Z(g)$ if and only if $g$ is in the principal ideal  of $f$ (Wojcicki theorem, see \cite{MS})  \end{proof} 

Recall that a rational simplex in $[0, 1]^X$, where $X$ is a finite set, is the convex
envelope of finitely many, affinely independent rational points. A rational
polyhedron is a finite union of rational simplexes.

The results of \cite{DL} show the relevance of theories in infinitely many variables, possibly finitely axiomatized. 
In order to treat with infinite dimensional objects, we give the following definition taken from \cite{DL}:
\begin{definition} Let $X$ be a set, possibly infinite. We define a  cylinder
(rational) polyhedron in a hypercube $[0, 1]^X$ a subset of $[0, 1]^X$ of the form
$$C_X(P_0) = \{f\in [0, 1]^X |\ f|_Y\in P_0\}$$
where $Y$ is a finite subset of X and $P_0 \subseteq [0, 1]^Y$ is a rational polyhedron.
Here $f|_Y$ denotes the function f restricted to Y , that is, $f|_Y = f \circ j$, where
$j : Y\rightarrow X$ is the inclusion map.\end{definition}

Like rational polyhedra are zerosets of McNaughton functions in finite dimension, cylinder rational polyhedra are the zerosets of McNaughton functions in possibly infinite dimension. 

Now we can give the following properties of spectra of free MV-algebras. 

\begin{theorem}\label{thm:kospec} Let $A=Free(k)$  be a free MV-algebra over $k$ generators. Then $Spec(A)$ is spectral. Moreover, there is a lattice isomorphism between  the lattice $\overset{\circ}{K}(Spec(A))$ of  the  compact open subsets  of $Spec(A)$ and  the lattice of the cylinder rational polyhedra of 
$[0,1]^k$.
\end{theorem}
\begin{proof} (i) We will prove that $Spec(A)$ is compact.

If $I_\alpha$ is a collection of ideals of $A$ where
$\cap V (I_\alpha) = \emptyset=V (A)$,
then by Lemma \ref{lemm} $$\oplus I_\alpha=A$$

Since $1\in A$, we have  $1=i_{\alpha_1}\oplus i_{\alpha_2}\oplus \dots\oplus i_{\alpha_k}$, where $i_{\alpha_i}\in I_{\alpha_i}$,  for $i=1,\dots, k$. From this $1\in\oplus I_{\alpha_i}$, or equivalently

   $$V(\oplus_{i=1}^k I_{\alpha_i})=\cap_{i=1}^k  V (I_{\alpha_i}) = V (A)=\emptyset$$ which concludes the proof.
   
    (ii) It has a basis of open compact subsets: See Proposition \ref{pro}. 
   
   (iii) From (i)-(ii) and Proposition \ref{pr} we can conclude that $Spec(A)$ is spectral.
   
   (iv) The last assertion derives from Proposition \ref{prof} 
   \end{proof}  \vskip5mm \vskip5mm
 
In order to treat quotient maps we make use of what follows:

\begin{notation} Let $I$ be an ideal of $A$, then $h_I$ denotes the canonical  correspondence from $A$ to $A/I$ such that $x\mapsto  x/I$. \end{notation}

  \begin{lemma}  (i) Let $h:A\rightarrow B$ be a   surjective homomorphism between two MV-algebras and $I$ be an ideal, then $h(I)$ is an ideal.
 
 (ii) Let $h:A\rightarrow B$ be a homomorphism between two MV-algebras and $P'$ be a prime ideal of $B$, then $h^{-1}(P')$ is a prime ideal of $A$. 
 
 (iii) Let $h$ be surjective,  $P$ be a prime ideal of $A$ which contains $Ker h$, then $h(P)$ is a prime ideal of $B$.
 \end{lemma}
 
  \begin{proof} (i) Obviously, $h(I)$ is closed under sum.
 
 We now prove closure downwards: If $x \in B,\, y \in  h(I)$ and $x\leq y$, then $x\in h(I)$. 

Since $y\in h(I)$, there exists $i\in I$ such that $y=h(i)$. Since $h$ is surjective, there exists $j\in A$ such that $x=h(j)$. 

From $h(j)\leq h(i)$, we derive  $x=h(j)=h(j\wedge i)$, where $j\wedge i\in I$, whence the thesis comes.
  
    (ii) Clearly $h^{-1}(P')$ is an ideal. We prove that the ideal $h^{-1}(P')$ is prime.

Suppose that  $a\wedge b\in h^{-1}(P')$ for $a, b\in A$. Then we have $h(a\wedge b) \in P'$.
Since $h$ is a homomorphism, we obtain
\begin{align*} h(a)\wedge h(b)=h(a\wedge b)\in P'.
\end{align*}

Since $P'$ is a prime ideal, it follows that either $h(a)\in P'$ or $h(b)\in P'$.
Hence we have either $a\in h^{-1}(P')$ or $b\in h^{-1}(P')$.
This proves that the ideal $h^{-1}(P')$ is prime.
  
  (iii) First observe that, since $h$ is a surjective homomorphism and $P$ is an ideal, we get that  $h(P)$  is an ideal. 
  
  We prove  now that the ideal $h(P)$ is prime. 
  
  Suppose  that  $r\wedge s \in h(P)$,  for $r,s\in B$.  By the surjectivity of $h$ we have $r=h(a),\, s=h(b)$ for some $a,\,b\in A$. 
Now $h(a)\wedge h(b)=h(c)$, for some $c\in  P$,  and this  becomes $(a\wedge b)\ominus  c\in ker h$. Since  $P\supseteq ker h$, we can conclude $(a\wedge b)\ominus c\in P$, but $a\wedge b\leq c\oplus ((a\wedge b)\ominus c)$ and the right hand side is in $P$, 
so also $a\wedge b\in P$. Since $P$ is prime,  either $a\in P$ or $b\in P$. It follows  either $r\in h(P)$ or $s\in h(P)$.  This proves that the ideal $h(P)$ is prime. \end{proof} 
 
\begin{proposition}\label{i}  (i)  Let $J$ be an ideal of an MV-algebra A. Then the map
$I \rightarrow  h_J (I)$ determines an inclusion preserving one-to-one correspondence between
the ideals of A containing J and the ideals of the quotient MV-algebra $A/J$. 

(ii) The inverse map also preserves inclusions.  

(iii)   $Spec(A/J)$ and $V(J)$, both endowed with the Zariski topology,  are  homeomorphic. 
\end{proposition}
\begin{proof} (i), (ii)  The first two statements are contained in \cite[Proposition 1.15] {M}.

(iii) For the last statement: Since  the prime ideals  of $A/J$ are precisely the ideals $P/J$ for prime ideals $P\in V(J)$, the thesis follows.\end{proof}

The map $$h_*:Spec(G/J)\rightarrow Spec(F/I) $$ 
defined by $h_*(P):=h^{-1}(P)$, for every prime ideal $P$, is continuous.

Indeed,  
If $ V(T)$ is  some closed set of
$Spec(F/I)$, then $h_*^{-1}(V (T)) = V (h(T))$ and the latter set is also closed. 

Let us prove the last equality:

$$h_*^{-1}(V (T))=\{P\in Spec(G/J):  h_*(P)\in V(T)\}$$
$$=\{P\in Spec(G/J): T\subseteq h_*(P)\}=\{P\in Spec(G/J):  h(T)\subseteq P\}=$$
$$V (h(T)).$$

 \section{A  characterization of the spectrum of free MV-algebras}\label{Spectrum}

In this section we  characterize the spectra of free MV-algebras, in this way we are able to characterize the spectra of any MV-algebra, as the spectra of MV-algebras are exactly the closed subsets of the spectra of free MV-algebras (via the homeomorphism introduced in Proposition \ref{i}). 
In the next theorem we will characterize the spectra of free MV-algebras in purely topological means. To this aim we introduce an ad hoc class of topological spaces called the cylinder based spaces.
\begin{definition}
A topological space $X$ is called cylinder based if: 
\begin{enumerate}
\item $X$ is spectral;
\item the lattice $\overset{\circ}{K}(X)$ is anti isomorphic to the lattice $Cypol(k)$ of cylinder rational polyhedra of $[0,1]^k$ for some cardinal $k$, possibly infinite.
 \end{enumerate}\end{definition}

Now we can give the following characterization of spectra of free MV-algebras. 

\begin{theorem}\label{thm:fond}  A topological space $X$ is homeomorphic to $Spec(Free(k))$ for some $k$ if and only if it is cylinder based.
\end{theorem} 

\begin{proof} By Theorem \ref{thm:kospec}, Spec is spectral, and  its compact open sets form a lattice anti isomorphic to the lattice $Cypol(k)$.

Conversely, suppose these two conditions hold of a space $X$. By Stone duality between spectral spaces and lattices, see \cite{S1}, two spectral spaces with isomorphic lattices of compact open sets are homeomorphic. So $X\cong Spec(Free(k))$.\end{proof} 

By the previous theorem we derive the following criterion.

\begin{theorem} A topological space is the spectrum of an MV-algebra if and only if it is a closed subset of a cylinder based space.\end{theorem} 

It is remarkable that $Max(Free(k))$ is a space different from $Spec(Free(k))$ (for instance, it is Hausdorff), but has an isomorphic basis of open sets, given by the complements of zerosets of McNaughton functions. However, spectral spaces with isomorphic bases of compact open sets are homeomorphic by Stone duality. 

\medskip

 We can give a  characterization of $Max(Free(k))$ partly parallel to theorem \ref{thm:fond} as follows:

\begin{theorem} \label{thm:max} A topological space $X$ is homeomorphic to $Max(Free(k))=[0,1]^k$ if and only if:
\begin{itemize}
\item $X$ has an open basis $B$ which is a lattice anti isomorphic to the lattice $Cypol(k)$. 
\item Let us denote by $C$ the set of complements of $B$. If $Y$ is an ideal in $C$ and $c\in C$, then
$\bigcap Y\subseteq c$ holds if and only if for every $g\in B$ with $c\cap g=\emptyset$ there is $y\in Y$ such that $y\cap g=\emptyset$.  
\item Singletons of elements of $X$ coincide with nonempty minimal intersections of complements of elements of $B$.
 \end{itemize}\end{theorem}

\begin{proof} Let us show that $[0,1]^k$ satisfies the three conditions. 

The first  point follows because $B$ is the set of the complements of the zerosets of McNaughton functions in $[0,1]^k$.  

To prove the second, fix $Y$ and $c$. 

Suppose $\bigcap Y\subseteq c$ and $c\subseteq \neg g$, that is $c\cap g=\emptyset$. 

Then $\bigcap Y\subseteq \neg g$. By compactness, there is $y\in Y$ such that $y\cap g=\emptyset$. 

Conversely, suppose that for every $g\in B$ with $c\cap g=\emptyset$ there is $y\in Y$ such that $y\cap g=\emptyset$.  Then 
$g=\neg c$ gives some $y\in Y$ such that $y\cap \neg c=\emptyset$ so $\bigcap Y\cap\neg c=\emptyset$. 

To prove the third point, every singleton is an intersection of basic closed sets, and conversely, a minimal nonempty intersection of basic closed sets is a singleton, otherwise if we have two points we can separate them with a basic closed set because the space $Max(Free(k))$ is Hausdorff.

Conversely, suppose $X$ is a topological space with the three properties listed above. We define a homeomorphism between $X$ and $Max(Free(k))=[0,1]^k.$ 

Let $B=\{b_f\}$ be the basis of $X$ indexed by $f\in Free(k)$. Let $Y\subseteq Free(k)$.  The homeomorphism $h:X\to Max(Free(k))$ sends any element $p$ of $X$ such that $$\{p\}=\bigcap_{f\in Y}\neg b_f$$ to the element $h(p)$ in Max such that $$\{h(p)\}=\bigcap_{f\in Y}Z(f).$$  

We note first that the function $h$ is well-defined. In fact, quite generally, for every $Y,Z,Y',Z'$,  
 if $$\bigcap_{f\in Y}\neg b_f=\bigcap_{f\in Y'}\neg b_f$$ in $X$, then 
$$\bigcap_{f\in Y}Z(f)=\bigcap_{f\in Y'}Z(f)$$ in $Max(Free(k))$ and conversely. This follows from the third point. So, the definition of $h$ does not depend on how the singleton of $p$ is represented as intersection of basic closed sets. Moreover, 
 suppose $\bigcap_{f\in Y}\neg b_f$ is a singleton.  By the third point, it is a nonempty minimal intersection of basic closed sets. 
So, also the corresponding intersection $\bigcap_{f\in Y}Z(f)$ is  a nonempty minimal intersection of basic closed sets. 

 Since two points of $Max(Free(k))$ can always be separated by two basic closed sets, the latter is a singleton in $Max(Free(k))$. So the function $h$ is uniquely determined and sends singletons to singletons. 

Moreover, $h$ is injective, as follows from the previous equalities,  and $h$ is surjective, too; hence $h$ is a bijection. In fact, every singleton of $Max(Free(k))$ has the form $\bigcap_{f\in Y}Z(f)$, and is a nonempty minimal intersection of basic closed sets.  Hence also   $\bigcap_{f\in Y}\neg b_f$ is a nonempty minimal intersection by the third point. Hence the latter intersection is a singleton by the third point.

Moreover, the preimage of $Z(f)$ is $\neg b_f$ and conversely, so the preimages of basic closed sets under $h$ are basic closed, and the same holds for images. So $h$ and $h^{-1}$ are continuous, and $h$ is a homeomorphism. 
\end{proof} 

\section{On spectra of MV-algebras and functors}\label{MV}

By the previous duality, spectra of MV-algebras coincide with closed sets of spectra of free MV-algebras. So, if $I$ is an ideal of $Free(k)$, then the quotient MV-algebra $Free(k)/I$ has spectrum isomorphic to $V(I)$, and 
the theorems of the previous section can be relativized to the ideals $I\subseteq Free(k)$. 

\begin{theorem}\label{thm:I} A topological space $X$ is homeomorphic to $Spec(Free(k)/I)$ if and only if:
\begin{itemize}
\item $X$ is spectral;
\item  $\overset{\circ}{K}(X)$  is a lattice anti isomorphic to the lattice $Cypol(k,I)$  of cylinder polyhedra of $[0,1]^k$ intersected with $Z(I)$.
\end{itemize}
\end{theorem} 

\begin{proof} $Spec((Free(k))/I)$ is homeomorphic with $V(I)$, and the basic opens of $V(I)$ are of the form $O(f)\cap V(I)$, where $f\in Free(k)$. The isomorphism sends $O(f)\cap V(I)$ to $Z(f)\cap Z(I)$.  
\end{proof} 

\begin{corollary} Let $I,J$ be ideals of $Free(k).$ Then the following conditions are equivalent:
\begin{itemize}
\item $I\subseteq J$;
\item $V(J)\subseteq V(I)$. 
\end{itemize}\end{corollary} 

\begin{proof} This follows from the fact that $V(I)$ is the set of prime ideals which contain $I$ and $I=\bigcap V(I)$.
 \end{proof} 

\subsection{Some functors} \label{top}

 Let $MV_p$ be  the category  of presented MV-algebras,
i.e. the category whose objects are pairs $(F,I)$, where 
 $F$ is an MV-algebra freely generated  and $I$ is an ideal of $F$; a morphism $f:(F,I)\to (G,J)$ is an MV-algebra homomorphism sending $F$ to $G$ and $I$ to $J$. Note that MV-algebra homomorphisms from $F$ to $G$ coincide with tuples of MV-polynomials since $F$ and $G$ are free. 

Intuitively, the pair $(F,I)$ is a presentation of the MV-algebra $F/I$. Note that every MV-algebra admits such a presentation. Note also that the category $MV_p$ is equivalent to MV (assuming the axiom of choice), see \cite{MS}. 

Now we wish to relate $MV_p$ with a category built on spectral spaces. The target category is what we call the category $SIC$ of ``spectra in context''. Intuitively, we have two spectral spaces one corresponds to the spectrum of a free MV-algebra, and the other is embedded in the first. However, rather than sticking to inclusions, we consider a certain kind of monomorphisms.  This will allow us to construct some categorial equivalences. 

Formally, the objects of SIC are triples $(X,C,m)$ where $X$ is  a cylinder based topological space, $C$ is a spectral space, and $m:C\to X$ is a  monomorphic spectral map preserving closed sets. A morphism from $(X,C,m)$ to $(Y,D,n)$ is a pair of spectral maps $s:X\to Y$, $t:C\to D$ which makes the corresponding square commute.

Recall that cylinder based spaces coincide with spectra of free MV-algebras, see Theorem \ref{thm:fond}. The relation of $MV_p$ and $SIC$ is  as follows. 

\begin{lemma} Every object $(X,C,m)$ in $SIC$ is isomorphic to $(X,m(C),j)$ where $j$ is the inclusion from $m(C)$ to $X$.
\end{lemma} 

\begin{proof} The pair of arrows giving the isomorphism is $(m,id_X)$ where $id_X$ is the identity on $X$. 
\end{proof} 

\begin{theorem}\label{thm:sic} There is a contravariant functor $\alpha$ from the category $MV_p$ to the category $SIC$, injective on objects, and surjective on objects up to isomorphism.
\end{theorem}

\begin{proof} Let $(F,I)$ be an object of $MV_p$, then  $\alpha(F,I):=(Spec(F),V(I),j)$ where $j$ is the inclusion from $V(I)$ to $Spec(F)$. Note that if $Spec(F)=Spec(F')$ then $F=F'$ and if $V(I)=V(I')$ then $I=\bigcap V(I)=\bigcap V(I')=I'$. 

Finally $\alpha$ is surjective on objects up to isomorphism because every object of $SIC$ is isomorphic to one of the form $(Spec(F),V(I),j)$.
\end{proof}

It is useful to compare our functor $\alpha$ with \cite{MS}. In that paper, up to a different notation, there is essentially a duality  between a category $MS_1$ of presented semisimple MV-algebras, viewed as pairs $(F,I)$ such that $F$ is free and $F/I$ is semisimple,  and a category $MS_2$ of closed subsets of $[0,1]^k$, a topological space homeomorphic to the maximal spectrum of $Free(k)$ (see \cite[Proposition 4.2]{M}) . So $MS_1$ has as objects the pairs $(Free(k),I)$  where $I$ is an ideal of $Free(k)$ and $Free(k)/I$ is semisimple: equivalently, $I$ is an intersection of maximal ideals of $Free(k)$. 

The category $MS_1$ is analogous to $MV_p$, and the category $MS_2$ is partially analogous to SIC. The functor $\alpha$ is analogous to the one sending $(Free(k),I)$ to $(Max(Free(k)),V(I)\cap Max(Free(k)))$.

 \begin{remark} Let $h:(F,I)\to (G,J)$ be a morphism in $MV_p$. Thanks to Proposition \ref{i} and the homeomorphisms introduced therein,  we get that  $$\alpha(h):\alpha(G,J)\rightarrow \alpha(F,I)$$ gives a continuous function from $V(J)\subseteq Spec(G)$ endowed with the induced  Zariski topology to $V(I)\subseteq Spec(F)$ endowed with the induced Zariski topology.   \end{remark}

We note finally that we can compose the functor $\alpha$ from $MV_p$ to SIC 
 with the Mundici functor $\Gamma$ (up to composing with an equivalence from MV to $MV_p$) and obtain a functor from abelian $\ell$-groups with unity to SIC. 

\subsection{Presented $\ell$u-groups}

In the previous subsection we have defined a category $MV_p$ of presented MV-algebras and a functor from this category to another category SIC related to spectral spaces. Given the Mundici equivalence between MV-algebras and $\ell$u-groups, we can perform a similar construction on $\ell$u-groups, i.e. $\ell$-groups with strong unit $u$. To this aim, we introduce the category  $Lu_p$ of presented $\ell$u-groups. Let $G_k=\Gamma^{-1}(Free(k))$, where $Free(k)$ is the free MV-algebra over $k$ elements and $\Gamma^{-1}$ is the inverse Mundici functor. The objects of $Lu_p$ are pairs $(G_k,I)$ where $I$ is an $\ell$-ideal of $G_k$. 

Intuitively, $(G_k,I)$ represents the  $\ell$-group $G_k/I$, which is unital because unital groups are closed under quotients. 

Morphisms from $(G_k,I)$ to $(G_h,J)$  are unital $\ell$-group morphisms from $G_k$ to $G_h$ which send $I$ to $J$. 

Like the case of MV-algebras, we have:
\begin{lemma} Every $\ell$u-group admits a presentation.\end{lemma} 
\begin{proof} Let $(G,u)$ be an $\ell$u-group and $A=\Gamma(G,u)$. Let $A=Free(k)/I$ be a presentation of the MV-algebra $A$. Let $\pi:Free(k)\to Free(k)/I$ be the quotient map. Note that $\pi$ is surjective. So, also $\Gamma^{-1}(\pi):\Gamma^{-1}(Free(k))\to \Gamma^{-1}(A)$ is surjective by \cite[Lemma 7.2.1]{CDM}. So $(G,u)$ is an epimorphic image of $\Gamma^{-1}(Free(k))$ and admits a representation.  \end{proof} 

\begin{corollary} The categories $MV_p$  and $Lu_p$  are equivalent.
\end{corollary} 

\begin{proof} The equivalence sends $(Free(k),I)$ to $(\Gamma^{-1}(Free(k)),\Gamma^{-1}(I))$. \end{proof} 

\begin{corollary} There is a functor from $Lu_p$ to $SIC$ injective on objects. \end{corollary}

\subsection{On the functors $\beta$, $Id_c$, $Spec$ and $\overset{\circ}{K}$}

Let $A$ be an MV-algebra. Let $\beta(A)$ the Belluce lattice of $A$ from \cite{Palermo}; that is $\beta(A)=A/\sim$, where $x\sim y$ if $x$ and $y$ belong to the same prime ideals.  Further  let $Id_c(A)$ be the lattice of principal ideals of $A$. Both $\beta$ and $Id_c$ extend to functors from the category of MV-algebras to the category of bounded distributive lattices. 

Interestingly, the prime ideals of the lattice $\beta(A)$ coincide with the prime ideals of the MV-algebra $A$ (see \cite[on p.375]{Palermo}). Moreover, let 
$\overset{\circ}{K}(X)$ be the lattice of compact open sets of a spectral space $X$. We have:

\begin{lemma} (see \cite{W19}) For every MV-algebra $A$, $\overset{\circ}{K}(Spec(A))\cong Id_c(A)$.\end{lemma} 

\begin{lemma} $Id_c(A)\cong\beta(A)$.\end{lemma} 

\begin{proof} This holds because two elements of $A$ belong to the same prime ideals if and only if they generate the same ideal. \end{proof} 

\begin{corollary} $\overset{\circ}{K}(Spec(A))\cong\beta(A)$.
\end{corollary}

One can ask whether the three functors $\beta,Id_c,\overset{\circ}{K}\circ Spec$ are naturally isomorphic.

It follows from  \cite{DL} that:
\begin{itemize}
\item when $A=Free(k)$ is free, $\beta(A)$ is the lattice of cylinder polyhedra in $[0,1]^k$;
\item the lattices $\beta(A)$, for $A$ arbitrary, are exactly the closed epimorphic images of $\beta(Free(k))$ for some $k$.   
\end{itemize}

\subsection{A new category of lattices}

We have defined two categories $MV_p$, based on MV-algebras, and SIC, based on spectral spaces. It is natural to try to relate these categories with analogous  categories of lattices and Priestley spaces, using Stone duality for lattices and Priestley duality. 
We have defined two categories $MV_p$, based on MV-algebras, and SIC, based on spectral spaces. It is natural to try to relate these categories with analogous  categories of lattices and Priestley spaces, using Stone duality for lattices and Priestley duality. 

So we introduce a category LIC (lattices in context)  analogous to the category SIC previously introduced. 

Namely, LIC is the category of triples 

$$(\beta(Free),L,f)$$ where $\beta$ denotes the Belluce functor, $L$ is a lattice and $f:\beta(Free)\to L$ is a lattice epimorphism such that 
the Stone dual of $f$  preserves closed sets.  A morphism from $(\beta(Free_1),L_1,f)$ to $(\beta(Free_2),L_2,g)$ is a pair of lattice homomorphisms $s:\beta(Free_1)
\to\beta(Free_2)$, $t:L_1\to L_2$ which makes the corresponding square commute. 

Now we can state:

\begin{theorem}\label{thm:siclic} The categories SIC and LIC are dual. 
\end{theorem} 

\begin{proof} There exists a  duality $\gamma$ which assigns to every object $(\beta(Free),L, f)$  of LIC an object of SIC, namely 

\noindent $\gamma(\beta(Free), L, f)$=(X=Stone dual of $\beta(Free)$, C=Stone dual of $L$, $m$=Stone dual of $f)$ 

Indeed, by Theorem \ref{thm:fond} the Stone dual of $\beta(Free)$ is a cylinder based topological space and by Stone duality $C$ is a  spectral space,  and  $m$ is a monomorphic spectral map which preserves closed sets. 

Moreover $\gamma$ assigns to    a morphism from $(\beta(Free_1),L_1,f)$ to $(\beta(Free_2),L_2,g)$  a morphism from  $(X=$Stone dual of $\beta(Free_1)$, $C=$Stone dual of $L_1$, $m=$Stone dual of $f)$ to $(Y=$Stone dual of $\beta(Free_2)$, $D=$Stone dual of $L_2$, $n=$Stone dual of $g)$. 

$\gamma$ is a duality thanks to the properties of Stone duality. \end{proof}

Note the following characterization of the functions $f$ involved in the duality:

\begin{lemma}\label{lemma:preserv} Let $f:L\to L'$ be a lattice homomorphism. The Stone dual of $f$ preserves closed sets if and only if for every prime ideal $P$ of $L$ and for every ideal $I$ of $L'$ we have,
\begin{equation} \label{eq:cl} P\supseteq f^{-1}(I)\Rightarrow P=f^{-1}(Q)\ for\ some\ prime\ Q\supseteq I.\end{equation}
\end{lemma} 

\begin{proof} For a less cumbersome notation we may write $VI$ rather than $V(I)$. 

Denote by $\overline{X}$ the closure of a set in $Spec(L)$ or $Spec(L')$. Note that $\overline{X}=V\bigcap X$. Let $m:Spec(L')\to Spec(L)$ be the Stone dual of $f$. We note that $m$ preserves closed sets if and only if for every $X\subseteq Spec(L)$ we have 
$$(1){m}(\overline{X})=\overline{{m}(X)}$$
where ${m}(X)$ is the image of $X$ under $m$.

Now (1) is equivalent to 
 $$(2){m}(V\bigcap X)=V\bigcap{{m}(X)}$$
or also 
 $$(3){m}(V\bigcap X)=V\bigcap_{y\in X}m(y)$$
and since $m=f^{ -1}$ commutes with intersections we have
$$(4) {m}(V\bigcap X)=Vm(\bigcap_{y\in X}y)=Vm(\bigcap X)=Vf^{-1}(\bigcap X).$$
Now $I=\bigcap X$ is an ideal of $L'$, and all ideals have this form in the sense that $I=\bigcap V(I)$. So we can replace $\bigcap X$ with $I$ and for all ideals $I$ we have
$$(4) {m}(V(I))=V(f^{-1}(I)).$$
So (4) means that both members contain the same primes $P$ of $L$, so 
$$(5) P\in V(f^{-1}(I))\Rightarrow P\in {m}(V(I))$$
that is
$$(6) P\in V(f^{-1}(I))\Rightarrow P=m(Q),Q\supseteq I$$
or, replacing $m$ with $f^{-1}$, and by definition of $V$,
$$(7) P\supseteq f^{-1}(I)\Rightarrow P=f^{-1}(Q),Q\supseteq I$$

which is equivalent to the formula in the statement  (note that the converse of (7) is trivially true).\end{proof} 

\subsection{A new category of Priestley spaces}

A further functor, studied in \cite{Palermo}, goes from the category of MV-algebras to the category of Priestley spaces, which are certain partially ordered topological spaces. More precisely, a Priestley space is an ordered topological space $(X,\tau,\leq)$ such that $(X,\tau)$ is compact, and if $x$ is not less than $y$, then there is a clopen upset $U$ such that $x\in U$ and $y\notin U$. The Priestley dual of a lattice $L$ is the prime spectrum of $L$ with the inclusion order and patch topology, whose basic opens are $\{P\in Spec(L)|a\in P\}$ and  $\{P\in Spec(L)|a\notin P\}$ for $a\in L$.  On morphisms, the Priestley dual of a homomorphism $f:L\to L'$ is $f^{-1}:Spec(L')\to Spec(L)$. 

The idea in \cite{Palermo} is first to apply the functor $\beta$ from MV-algebras to bounded distributive lattices, where $\beta(A)$ is the Belluce lattice of $A$,  and then to apply Priestley duality. Note however that at the time of writing \cite{Palermo} there was no characterization of the range of the functor $\beta$, whereas now the situation is clarified, for instance, by \cite{DL}. 

In the previous subsections we introduced a lattice-related category LIC. In analogy we can introduce a category PIC (Priestley in context) related with Priestley spaces. Objects are triples $(\Xi_1,\Xi_2,g)$ where $\Xi_1$ is the Priestley dual of the Belluce lattice of a free MV-algebra, $\Xi_2$ is another Priestley space and $g: \Xi_2\to \Xi_1$ is an injective Priestley map which preverses closed upsets. 

Morphisms in PIC from $(\Xi_1,\Xi_2,g)$ to $(\Xi'_1,\Xi'_2,g')$ are pairs of Priestley maps which make the corresponding diagram commute. 

The following lemma relates LIC and PIC.

\begin{lemma} Let $L$ be a lattice. Closed upsets in the patch topology of $Spec(L)$ coincide with closed sets in the Zariski topology of $Spec(L)$. 
\end{lemma} 

\begin{proof} It follows from the well known isomorphism between the category of Priestley spaces and the category of spectral spaces, see \cite{Bez}.  
\end{proof} 

By this lemma, a map from $Spec(L)$ to $Spec(L')$ preserves closed sets in the Zariski topology of $Spec(L)$ if and only if it preserves closed upsets in the patch topology of $Spec(L)$. 

Now we can apply Priestley duality in the same way as we used Stone duality in Theorem \ref{thm:siclic} and we have:

\begin{theorem}\label{thm:Pries}  Priestley duality gives a duality between $LIC$ and $PIC$. 
\end{theorem} 

In order to summarize the previous results 
we draw the following diagram, where MV is the category of MV-algebras, FG is the category of ``MV-spaces'' contained in \cite{FG95} and DP the category again of ``MV-spaces'' in \cite{DP}, and $\longleftrightarrow^e$ denotes an equivalence, and $\longleftrightarrow$ a duality, and $\alpha$ is the functor described in Theorem \ref{thm:sic}.

$$\xymatrix{& DP  \arrow[r] & MV\arrow[l]\arrow[d]^\alpha\arrow[r] & FG\arrow[l]^e \\
& &  SIC\arrow[d] & \\
& & LIC\arrow[u]\arrow[d] &\\
&& PIC\arrow[u]}$$

\section{More on lattice homomorphisms}\label{sec:lathom}

In the previous sections we used repeatedly lattices and their maps (homomorphisms). In this section we expand on lattice homomorphisms and we give a  characterization of closed surjective lattice homomorphisms in the sense of \cite{W19}, which play an important role in \cite{DL}. 

\subsection{A characterization of closed surjective homomorphisms}

In this subsection we give a simple characterization of closed surjective lattice homomorphisms in the sense of \cite{DL} and \cite{W19}. 

If $S$ is a subset of a lattice we denote by $DS$  the downset of  $S$. If $B,C$ are subsets of a lattice we denote also $B\vee C=\{b\vee c|b\in B,c\in C\}$. 

\begin{theorem}\label{thm:closed} A surjective homomorphism $f:L\to L'$ between bounded distributive lattices is closed if and only if for every $b,c\in L'$ we have  

$$f^{-1}(Db\vee Dc)\subseteq  f^{-1}(Db)\vee  f^{-1}(Dc)$$
  
\end{theorem} 

\begin{proof} Remember from \cite{W19} that $f:L\to L'$ is closed whenever:

if 1) $f(a_0)\leq f(a_1)\vee c$

then 2) there is $x\in L$ such that $a_0\leq a_1\vee x$ and $f(x)\leq c$.

\medskip

In terms of downsets, closedness is equivalent to:

$$f^{-1}(D(f(a_1)\vee c)\subseteq \{y|\exists x\,y\leq a_1\vee x\ AND\ f(x)\leq c\}$$

or equivalently, using a union at the right hand side,

$$f^{-1}(D(f(a_1)\vee c))\subseteq \bigcup_{x\in f^{-1}(Dc)}D(a_1\vee x)$$

if $f$ is surjective, then every element $b\in L'$ is $f(a_1)$ for some $a_1\in L$, so we can quantify over $b\in L'$:

$$f^{-1}(D(b\vee c))\subseteq \bigcup_{x\in f^{-1}(Dc)}\bigcup_{a_1\in f^{-1}(b)}D(a_1\vee x)$$

or more compactly 

$$f^{-1}(D(b\vee c))\subseteq \bigcup_{x\in f^{-1}(Dc)}D( f^{-1}(b)\vee x)$$

$$f^{-1}(D(b\vee c))\subseteq D( f^{-1}(b)\vee  f^{-1}(Dc))$$

The next lemma is a property of distributive lattices:

\begin{lemma} $D(b\vee c)=Db\vee Dc$.\end{lemma}

\begin{proof} Suppose $x\in D(b\vee c)$. Then $x\leq b\vee c$, so $x=x\wedge (b\vee c)=(x\wedge b)\vee(x\wedge c)$, so $x\in Db\vee Dc$. 
The converse inclusion is obvious. 
\end{proof} 

Passing to the union we derive, for every two sets $B,C$: 

\begin{corollary} $D(B\vee C)=DB\vee DC$.
\end{corollary} 

 Using the previous corollary we have 

$$f^{-1}(Db\vee Dc)\subseteq D( f^{-1}(b))\vee D f^{-1}(Dc)$$
 
Now a key passage is:

\begin{lemma} $D( f^{-1}(Db))=f^{-1}(Db)$. That is, $f^{-1}(Db)$ is a downset. 
\end{lemma} 

\begin{proof} let $x\in f^{-1}(Db)$. Let $y\leq x$. Then $f(x)\leq b$ and $f(y)\leq b$, so $y\in f^{-1}(Db)$.
\end{proof}

By the previous results the inclusion becomes: 

$$f^{-1}(Db\vee Dc)\subseteq  f^{-1}(Db)\vee  f^{-1}(Dc).$$

\end{proof} 

Note that the characterization is given exclusively in terms of downsets and inverse images. Moreover the downset of an element $b\in L'$ is the intersection of all prime ideals containing $b$, so we can say that the condition is given in terms of prime ideals and inverse images. 

Actually the proof extends to arbitrary downsets rather than only principal downsets, so we have: 

\begin{corollary} A surjective homomorphism $f:L\to L'$ between bounded distributive lattices is closed if and only if for every $B,C\subseteq L'$ we have  

$$f^{-1}(DB\vee DC)\subseteq  f^{-1}(DB)\vee  f^{-1}(DC)$$
\end{corollary} 

\begin{proof} Suppose $f$ is closed. Let $x\in f^{-1}(DB\vee DC)$. Then $f(x)\in DB\vee DC$, so for some $c\in C,b\in B$ we have $f(x)\in Db\vee Dc$, 
so $x\in f^{-1}(Db\vee Dc)$, 
and  
by the theorem, $x\in f^{-1}(Dc)\vee  f^{-1}(Db)$ so $x\in f^{-1}(DC)\vee  f^{-1}(DB)$. The converse inclusion is obvious. \end{proof} 

Note that the inclusion can be replaced with equality: 
\begin{corollary}\label{cor:equality} A surjective homomorphism $f:L\to L'$ between bounded distributive lattices is closed if and only if for every $B,C\subseteq L'$ we have  

$$f^{-1}(DB\vee DC)=  f^{-1}(DB)\vee  f^{-1}(DC).$$
\end{corollary} 

\begin{proof} The first inclusion follows from the previous corollary. Let us prove the second. We note that $f^{-1}(DC\vee DB)$ is an ideal, since $DC\vee DB$ is an ideal and $f^{-1}$ preserves ideals. So from the clear inclusions $f^{-1}(DB)\subseteq f^{-1}(DC\vee DB)$ and $f^{-1}(DC)\subseteq f^{-1}(DC\vee DB)$ it follows $f^{-1}(DB)\vee  f^{-1}(DC)\subseteq   f^{-1}(DB\vee DC)$.

\end{proof} 

\begin{corollary}\label{cor:ideals} A surjective homomorphism $f:L\to L'$ between bounded distributive lattices is closed if and only if for every two ideals $I,J$ of $L'$ we have  

$$(*) f^{-1}(I\vee J)=  f^{-1}(I)\vee  f^{-1}(J).$$
\end{corollary}

\begin{proof} By the previous results, $f$ is closed if and only if (*) holds for every downset, or equivalently, for every principal downset. But every ideal is a downset and every principal downset is an ideal, so $f$ is closed if and only if (*) holds for ideals. \end{proof}

\subsection{Characterization of duals of closed maps}

We can characterize in a simple way Stone (or Priestley)  dual maps of closed epimorphisms between lattices. 

\begin{theorem} Let $L,L'$ be distributive lattices and $g:Spec(L')\to Spec(L)$. Then $g$ is the Stone (or Priestley) dual of a closed epimorphic map $f$ from $L$ to $L'$ if and only if:
\begin{itemize}
\item  $g$ commutes with arbitrary unions and intersections;
\item $g$ sends ideals to ideals;
\item for every $C,D$ closed subsets of $Spec(L')$ we have the equality

$$(1) \overline{{g}(C\cap D)}=\overline{{g}(C)}\cap\overline{{g}(D)}.$$

where $\overline{X}$ is the topological closure of a set $X$ and ${g}(X)$ is the image of $X$ under $g$.\end{itemize}
\end{theorem} 

\begin{proof} $g$ must commute with unions and intersections since it coincides with $f^{-1}$. Moreover (1) follows from a few calculations. In fact, first we note that for every set $X\subseteq Spec(L')$ we have $\overline{X}=V(\bigcap X)$. So the equality is equivalent to 

$$ (2) V\bigcap{{g}(C\cap D)}=V\bigcap{\hat{g}(C)}\cap V\bigcap{{g}(D)}.$$

By definition of $\hat{g}$ (2) becomes

$$ (3) V\bigcap_{P\in C\cap D}g(P)=V\bigcap_{P\in C}g(P)\cap V\bigcap_{P\in  D}g(P)$$

and by the properties of $V$ this becomes

$$ (4) V\bigcap_{P\in C\cap D}g(P)=V\bigcap_{P\in C}g(P)\vee\bigcap_{P\in  D}g(P)$$

where $X\vee Y$ is the ideal generated by $X\cup Y$. 

Since $C,D$ are closed we have $C=VI$ and $D=VJ$ for some ideals $I,J$, so (4) becomes

$$ (5) V\bigcap_{P\in VI\cap VJ}g(P)=V\bigcap_{P\in VI}g(P)\vee\bigcap_{P\in  VJ}g(P)$$

and since $g$ commutes with intersections (5) becomes 

$$ (6) Vg(\bigcap V(I\vee J)=V(g\bigcap VI\vee g\bigcap VJ)$$

Now we take intersections of both sides: 

$$ (7) \bigcap Vg(\bigcap V(I\vee J)=\bigcap V(g\bigcap VI\vee g\bigcap VJ)$$

Now for every ideal $X$, we have $X= \bigcap V(X)$, and since $g$ preserves ideals, the sets $g(I)$, $g(J)$ and  $g(I\vee J)$ are ideals. So we can simplify (7) as follows: 

$$ (8) g(I\vee J)=g(I)\vee g(J).$$

So (1) holds for every two closed sets $C,D$ if and only if (8) holds for every two ideals $I,J$. By corollary \ref{cor:ideals}, $f^{-1}$ is closed if and only if (8) holds for every ideals $I,J$. \end{proof} 

\subsection{On closed maps versus closed preserving maps}

We notice that, despite the similar name, being a closed surjective lattice homomorphism in the sense of \cite{W19} is not equivalent to being a surjective homomorphism    preserving closed  sets. Indeed every surjective lattice homomorphism between totally ordered lattices is  closed, by Theorem \ref{thm:closed}, but there are some which do not preserve closed sets. 

For instance, take $f:\{0,1,2,3\}\to\{0,1\}$, where we have two sets of numbers with their usual ordering,  such that $f(0)=0$ and $f(1)=f(2)=f(3)=1$. Then $f$ is increasing and surjective. However, $f$  does not satisfy the condition of Lemma \ref{lemma:preserv}, since the ideal $\{0,1,2\}$ contains $Ker\ f=f^{-1}(0)=\{0\}$ but it is not the inverse image of any ideal. So, $f$ does not preserve closed sets.

\section{Perfect MV-algebras and perfect ideals}\label{perfect}

This section,   which is  inspired by \cite{BDG}, is devoted to perfect MV-algebras.  We recall the definition of perfect MV-algebras and then we offer a characterization.

\begin{definition}An MV-algebra A is called perfect iff for every $x\in A$
 exactly one of $x$ and $\neg x$ is of finite order. \end{definition}
 
 Recall that for any $x\in A$, the order of $x$, in symbols $ord(x)$, is the smallest
natural number $n$ such that $nx =1.$ If no such $n$ exists, then we write $ord(x) =\infty$. 

Perfect MV-algebras are related to infinitesimals. We recall that an infinitesimal element in an MV-algebra is an element $a \not= 0$ such that $na\leq \neg a$, for every $n\in\N$, equivalently, iff $$na\ominus\neg a = n a\odot a = 0$$
where $a\odot b := \neg (\neg a\oplus \neg b)$ is  the \L ukasiewicz product. An element $a$ is coinfinitesimal if  $\neg a$ is infinitesimal. For any infinitesimal element $a$, the sequence
$$(0\leq a\leq 2a\leq 3a\leq\dots\leq na \leq\dots)$$ is strictly increasing. 

In the MV-algebra $[0, 1]$ and in all MV-algebras of functions taking values in $[0, 1]$, there are no such elements. On the other hand the Chang MV-algebra does have infinitesimals. 

 \begin{remark} \cite{BDL} An MV-algebra is perfect if and only if it is generated by the intersection of all
its maximal ideals.   More precisely, an MV-algebra A is  perfect if and only if it is nontrivial and $$A = Rad A\cup\neg Rad A$$ where  $Rad A$, the radical of $A$,   is the intersection of all maximal ideals of $A$ and  $\neg Rad A = \{x\in A : \neg x \in Rad A\}$.   \end{remark} 
Notice that  the non-zero elements of the radical of an MV-algebra coincide with the infinitesimals. So an MV-algebra is perfect if and only if it is generated by its infinitesimals. For instance, Chang MV-algebra is perfect. 

\vskip5mm
Now we offer a technical lemma.

\begin{lemma}(\cite[Proposition 5]{BDG}) An MV-algebra $A$ is perfect if and only if for every $a\in A$, $a^n=0$ for some $n\in\mathbb N$ if and only if $(\neg a)^m\not=0$ for every $m\in\mathbb N$. 
\end{lemma} 

\begin{definition} Let us call an ideal $I$ of an MV-algebra $A$ perfect if $A/I$ is a perfect MV-algebra.\end{definition}

From the previous lemma it follows:

\begin{corollary} An ideal $I$ of an MV-algebra $A$ is perfect if and only if  for every $a\in A$, $a^n\in I$ for some $n\in\mathbb N$ if and only if $(\neg a)^m\notin I$ for every $m\in\mathbb N$. \end{corollary}

\section{Spectra of $\V(C)$-algebras and functors}\label{closed}

In the previous section we studied perfect MV-algebras.  
In this section we study a larger class, that is, $\V(C)$-algebras, where 
 $C=K_1$ is  the Chang algebra, also known as  the first Komori algebra. 

Explicitly, $C$ is defined on the set
$$C = \{0, c, \dots, nc,\dots, 1-nc, \dots, 1 - c, 1\}$$
by the following operations (consider 0 = 0c): 

$x \oplus y =(m + n)c$ if $x = nc$ and $y = mc$;

 $x \oplus y=1 - (m - n)c$ if $x = 1 -nc$ and $y = mc$ and $0 < n < m;$
 
$x \oplus y=1 - (n - m)c$ if $x = nc$ and $y = 1 - mc$ and $0 < m < n$;

$x \oplus y=1$ otherwise;

$\neg x = 1 - nc$ if $x = nc,\; \neg x = nc$ if $x = 1 - nc$.
\vskip5mm

Now fix a cardinal $k$. We have seen that there is a correspondence between closed subsets of $Spec(Free(k))$ and MV-algebra quotients of $Free(k)$. We can study particular cases of this correspondence. 
For instance, let us consider MV-algebras $A$ in  the variety $\V(C)$ generated by  the  Chang algebra $C$, let us write down $A=Free(k)/I$.

\medskip

We can ask how the closed set $V(I)\subseteq Spec(Free(k))$ looks like. We begin with a lemma:

\begin{lemma}\label{sit}  A prime ideal $P$ of an MV-algebra $A$ is perfect if and only if  the maximal ideal $M_P$ containing $P$ is supermaximal, i.e. $A/M_P =\{0,1\}$.\end{lemma} 
\begin{proof}
Let $P$ be a prime ideal of $A$. Suppose $P$ is perfect. Then $A/P$ is perfect and totally ordered. Since $P\subseteq M_P$, we have a surjective homomorphism $A/P\to A/M_P$. Since $A/P$ is perfect and $M_P$ is maximal, $A/M_P$ is in $\V(C)$  and is simple. But the only simple $\V(C)$ algebra is $\{0,1\}$.

Conversely, suppose $A/M_P =\{0,1\}$. Then $A/P$ is totally ordered and there is a homomorphism $h$ from $A/P$ to  
$A/M_P=\{0,1\}$. Hence $A/P$ must be perfect, otherwise there is an element $a\in A/P$ with $na=1$ and $a^n=0$ for some $n$, and it is impossible both $h(a)=0$ (contrary to $h(na)=1$)  and $h(a)=1$ (contrary to $h(a^n)=0$). So, $P$ is a perfect ideal. 
\end{proof}

Now we can characterize $\V(C)$ in terms of ideals: 

\begin{theorem}\label{theor} Let $A=Free(k)/I$ be an MV-algebra. The following  items are equivalent:
\begin{itemize}
\item $A$ is in $\V(C)$;
\item every maximal ideal of $A$ is supermaximal (that is, its quotient is $\{0,1\}$).
\item every element of $Spec(A)$ is a perfect ideal.
\item every element of $V(I)$ is perfect.
\end{itemize} \end{theorem} 

\begin{proof} Let us first prove that the first two points are equivalent. 

If $A$ is in $\V(C)$, then all its quotients modulo prime ideals are in $\V(C)$ and are totally ordered. Notice  that every totally  ordered MV-algebra $B$ in $\V(C)$ is perfect. In fact, in every MV-algebra in $\V(C)$, every element is infinitely close to a Boolean. But in a totally ordered, MV-algebra, the only idempotent elements are $0$ and $1$, so every element  of $B$ is either infinitesimal or coinfinitesimal, so $B$ is perfect. Now, all maximal quotients are in $\V(C)$ and are simple, so all maximal quotients are $\{0,1\}$. 

Conversely, suppose every maximal ideal of $A$ is supermaximal. Let $P$ be a prime ideal of $A$. Let $M$ be the unique maximal ideal of $A$ such that $P\subseteq M$. Then we have a surjective homomorphism from $A/P$ to $A/M=\{0,1\}$. But $A/P$ is a totally ordered MV-algebra, and since it has a surjective homomorphism into $\{0,1\}$, it is perfect. So all prime quotients of $A$ are perfect, hence they are in $\V(C)$, and $A$ embeds in the product of its prime quotients, so also $A$ is in $\V(C)$. 

Now by the previous lemma, an element of $Spec(A)$ is perfect if and only if the corresponding maximal in $Spec(A)$ is supermaximal, if and only if the corresponding maximal 
in $V(I)$ is supermaximal. So the remaining two points are equivalent to the first two.\end{proof}

\begin{definition} A  closed set  $K$  of $Spec(A)$ is called  $\V(C)$-closed if and only if every prime ideal  $P\in K$ is a perfect ideal.\end{definition}  

From the previous theorem the following corollary immediately follow: 

\begin{corollary} A closed set $K$ is $\V(C)$-closed if and only if every maximal element of $K$ is supermaximal. 
\end{corollary}

As a corollary of the previous results we can consider the root system structure of the prime spectrum of any MV-algebra.  Recall that a root system  is a partially ordered set $P$ such that for every $x\in P$ the final segment $\{y\in P: y\geq x\}$ is totally ordered, and a root is a root system with the greatest element \cite{DG}.

\begin{corollary}
Let $A$ be an MV-algebra. Let $R$ be a root of $Spec(A)$ and $M_R$ be the unique maximal ideal of $R$. Then $M_R$ is supermaximal if and only if every element of $R$ is a perfect ideal. \end{corollary}

\begin{proof} Suppose $M_R$ supermaximal. Let $P\subseteq M_R$ be a prime ideal of $A$. Then $A/P$ is totally ordered MV-algebra with an onto map on $A/M_R=\{0,1\}$, so $A/P$ is perfect and $P$ is perfect.

Conversely, if every element of $R$ is perfect, then $M_R$ is maximal and perfect, so the quotient $A/M_R$ is simple and perfect, so $A/M_R=\{0,1\}$ and $M_R$ is supermaximal.  
\end{proof}

More generally, it would be interesting to characterize $Max_p(A)$, where $p$ is a prime number, meaning the set of maximal ideals $M$ such that $A/M=\L_p$, the chain with $p$ elements.

Other expected results are the answers to these questions: "Which are the closed sets of perfect MV-algebras? Which are the closed sets of local MV-algebras?"
etc. 
\vskip5mm
Recall that local MV-algebras are MV-algebras with only one maximal ideal that, hence,
contains all infinitesimal elements. This class of algebras includes MV-chains and perfect MV-algebras.
More precisely, 
\begin{definition} An MV-algebra $A$ is called local if it has only one maximal
ideal, coinciding with $\{ a\in  A | ord(a) = \infty\}$. \end{definition} 
Equivalently, an MV-algebra A is local if and only if for every $x \in A$,
either $ord(x) < \infty$ or $ord(\neg x) < \infty.$

We recall that a primary ideal is an ideal $I$ such that if $xy\in I$, then $x^n\in I$ or $y^n\in I$ for some $n\in\mathbb N$.  Now  we show, as 
a particular case of the correspondence closed subsets of $Spec(Free(k))$ and MV-algebra quotients of $Free(k)$, the correspondence between closed sets and local MV-algebras:

\begin{theorem}\label{thm:local} Let  $C=V(I)$ a closed subset of $Spec(Free(k))$. The following are equivalent: 
\begin{itemize}
\item $C$ is the spectrum of a local MV-algebra;
\item $C$ has only one closed point;
\item the intersection of $C$ is a primary ideal.  
\end{itemize}
\end{theorem} 

\begin{proof} The first two points are equivalent because maximal ideals are closed points in $Spec$. 

For the third point, we can suppose $I$ is an ideal. Then $V(I)$ corresponds to the MV-algebra $A/I$. It is known that $A/I$ is local if and only if $I$ is a primary ideal (see \cite[Theorem 2.1]{BDL}). Note that $I$ is the intersection of $V(I)$.  
\end{proof} 

First recall that an MV-algebra is semisimple if the intersection of its maximal ideals is
zero, now the semisimple case can be treated in partial analogy with the local case as follows:

\begin{theorem} $Free(k)/I$ is semisimple if and only if the intersection of the maximal elements of $V(I)$ is zero. 
\end{theorem} 

\begin{proof} This holds because an MV-algebra is semisimple if and only if the intersection of its maximal ideals is zero. 
\end{proof} 

\section{Spectra of free MV-algebras in $\V(C)$}\label{VC}

We want to draw an analogy between spectra of free MV-algebras and spectra of free MV-algebras in subvarieties. 

Consider for simplicity only the variety $\V(C)$. 

Since $\Delta(\mathbb R)$  generates $\V(C)$, the free $\V(C)$-algebra over a cardinal $k$  can be defined as the MV-algebra of polynomial functions from 
$\Delta(\mathbb R)^k$ to $\Delta(\mathbb R)$,  see  \cite[Theorem 8.1]{DLV}. 

In the spectrum of $Free(\V(C),k)$ the Zariski topology is generated by the basic opens $O(f)=\{P\in Spec|f\notin P\}$, where $f\in Free(\V(C),k)$. They coincide with compact open sets. On the other hand, for every $f$, the zeroset of $f$, called $Z(f)$, is a subset of $\Delta(\mathbb R)^k$, more precisely, what we call  a cylinder rational fan. 
Note that fans are usually defined in vector spaces, but we can define them in $\Delta(\mathbb R)^k$ as follows. 

\medskip

We recall that  a subset $\sigma$ of a finite dimensional real vector space $V$ is a cone if for each $x \in \sigma$
and for each scalar $\alpha > 0$, $\alpha x \in \sigma$. A cone  $\sigma$ is simplicial if it is generated by a
set $S$ of (finitely many) linearly independent vectors.  A face of a cone generated
by a set $S$ is a cone generated by a subset of $S$.

A fan $\phi$ in a vector space $V$ is a finite set of simplicial closed cones in $V$ such
that:
\begin{itemize}  
\item if $\sigma\in\phi$, then any face of $\sigma$ is in $\phi$;
\item the intersection of any two closed cones in $\phi$ is also a closed cone in $\phi$. 
\end{itemize}
We say that a fan $\phi$ is rational if every closed cone belonging to $\phi$ is defined by linear
inequalities with rational coefficients.
\vskip5mm

More generally, if $F$ is a finite set, a fan in $\mathbb R^F$ is defined as follows. 

Recall that a quadrant of $\mathbb R^F$  is the set of vectors of $\mathbb R^F$  where the signs of components are given (there are $2^{F}$ many quadrants in $\mathbb R^F$). 

A fan in $\mathbb R^F$  is a subset  of $\mathbb R^F$  such that   each radical class $r$  is in a natural bijection $\beta_r$ with a quadrant of $\mathbb R^F$ (We recall that a
radical class of an MV-algebra is  an equivalence class modulo the radical). 

  Now we define a fan of $\Delta(\mathbb R)^F$ to be a subset of   $\Delta(\mathbb R)^F$  such that  every radical class $r$  is the image via $\beta_r$ of a fan in the corresponding quadrant of $\mathbb R^F$. 

\begin{lemma}\label{lemma:zerofan} Zerosets of MV-polynomials in  $\Delta(\mathbb R)^F$ coincide with fans. 
\end{lemma} 

\begin{proof} It follows from the description of MV-polynomials in  $\Delta(\mathbb R)^F$ given in  \cite[Theorem 3.1]{DLV}  with $m=1$. 
\end{proof} 

Like the case for MV-algebras, results of \cite{DL} show the importance of theories in $\V(C)$ with infinitely many variables, possibly finitely axiomatized.  
In order to treat also the case of spaces of infinite dimension we stipulate:

\begin{definition}\label{def:cylfan} Let $I$ be an infinite set. We define a  cylinder
(rational) fan in a hypercube $\mathbb R^I$ a subset of the form
$$C_I(P_0) = \{f\in \mathbb R^I |\, f|_Y\in P_0\}$$
where $Y$ is a finite subset of $I$ and $P_0 \subseteq \mathbb R^Y$ is a rational fan.

Here $f|_Y$ denotes the function f restricted to Y , that is, $f|_Y = f \circ j$, where
$j : Y\rightarrow I$ is the inclusion map.\end{definition}

From Lemma \ref{lemma:zerofan} it follows:

\begin{corollary} Let $I$ be an infinite set. Zerosets of MV-polynomials in  $\Delta(\mathbb R)^I$ coincide with cylinder fans. 
\end{corollary}

Note that basic opens  form a lattice with respect to inclusion, in fact $O(f\vee g)=O(f)\cap O(g)$ and $O(f\wedge g)=O(f)\cup O(g)$. Likewise, zerosets form a lattice since $Z(f\oplus g)=Z(f)\cap Z(g)$ and $Z(f\wedge g)=Z(f)\cup Z(g)$. Moreover:

\begin{lemma}\label{lemma:ogzg} Let $f,g\in Free(\V(C), k)$. The following are equivalent: 
\begin{enumerate} 
\item $O(f)\subseteq O(g)$
\item $Z(f)\subseteq Z(g)$
\item $f\in ideal(g)$.
\end{enumerate}
\end{lemma} 

\begin{proof} By definition 3 implies 1. 

The implication  from 1 to 3 comes from the fact that, in every MV-algebra, any maximal ideal among those not containing a given element is prime.

Also, clearly 3 implies 2. 

We are left with showing that 2 implies 3. But this follows from \cite[Lemma 6.1]{DLV}. 
\end{proof} 

\begin{corollary} $Spec(Free(\V(C),I))$ has a basis of open sets which is a lattice anti isomorphic to the lattice of rational fans in 
$\Delta( \mathbb R)^I$. \end{corollary} 

Now we want to characterize the spectrum of free $\V(C)$-algebras in a way analogous to free MV-algebras. 

\begin{definition}
A topological space $X$ is called cylinder fan based if: 
\begin{enumerate}
\item $X$ is spectral,
\item  for some cardinal $k$, possibly infinite, $\overset{\circ}{K}(X)$ is a lattice anti isomorphic to the lattice $Cyfan(k)$ of cylinder fans  of $\Delta(\mathbb R)^k$.
 \end{enumerate}\end{definition}

With a proof similar to Theorem \ref{thm:fond} we can prove: 

\begin{theorem}\label{thm:fan}  A topological space $X$ is homeomorphic to $Spec(Free(\V(C))(k))$ for some $k$ if and only if it is cylinder fan based.
\end{theorem} 

\begin{corollary}\label{cor:vcspec} A topological space is the spectrum of a $\V(C)$-algebra if and only if it is a closed subset of a cylinder fan based space.
\end{corollary}

\subsection{A functor from $\V(C)$-algebras to topological spaces}\label{sect:equiv}

In the previous sections we introduced a category SIC of spectra in context, related to the category of presented MV-algebras. The same can be done for every Komori variety of MV-algebras, however here we limit ourselves to $\V(C)$. We can call $\V(C)_p$ the category of presented $\V(C)$- algebras, that is pairs $(F,I)$ where $F$ is a free $\V(C)$-algebra and $I$ is an ideal of $F$. 

 Then it is natural to consider the category $\V(C)$-SIC of $V(C)$-spectra in context, whose objects are triples $(X,C,m)$ where $X$ is a fan based space, $C$ is a spectral space and $m$ is a spectral monomorphism from $C$ to $X$ preserving closed sets. Note that fan based spaces coincide with the spectra of free $\V(C)$-algebras, and closed subsets of fan based spaces coincide with spectra of $\V(C)$-algebras.  Then, similarly to Section \ref{top}, we have a functor from $\V(C)_p$ to $\V(C)$-SIC. 

\section{The case of $\V(K_m)$}\label{komori}

What we did for the variety $\V(C)$ can be done for every variety $\V(K_m)$ for every $m$. In particular, here we generalize section 
\ref{VC}. 

First let us denote $Free(m,k)$ the free MV-algebra in $\V(K_m)$ over $k$ elements. Then let us denote 
$$\Delta_m(\mathbb R)=\Gamma(\mathbb Z\ lex\ \mathbb R,(m,0));$$ this is an MV-algebra which generates $\V(K_m)$ (see \cite[Theorem 2.1]{DLV}). We can say that $\Delta_m(\mathbb R)$  is the continuous analogue of $K_m$.

Hence, the free $\V(K_m)$ algebra over a set $I$ can be defined as an MV-algebra of polynomial functions from $\Delta_m(\mathbb R)^I$ to $\Delta_m(\mathbb R)$. 

Like in $\V(C)$ there is a lattice isomorphism between compact open sets in the Zariski topology over $Spec(Free(m,I))$ and zerosets of polynomials in $\Delta_m(\mathbb R)^I$; more precisely, zerosets of polynomials can be  characterized as what we call $m$-fans, and the definition of an $m$-fan is as follows. 

If $F$ is a finite set, then an $m$-fan in $\Delta_m(\mathbb R)^F$ is defined as follows. Each radical class $r$ of $\Delta_m(\mathbb R)^F$ is in a natural bijection $\beta_r$ with a union of quadrants of $\mathbb R^F$. Now an $m$-fan in $\Delta_m(\mathbb R)^F$ is a subset which, in every radical 
class $r$, is the image via $\beta_r$ of a fan in the corresponding union of quadrants. 

As usual we can define a cylinder $m$-fan as the natural generalization of $m$-fan in possibly infinite dimensions. Lemma \ref{lemma:ogzg} transfers verbatim. 

\begin{definition}
A topological space $X$ is called $m$-fan based if: 
\begin{enumerate}
\item $X$ is spectral; 
\item the compact open sets form a lattice anti isomorphic to the lattice $mfan(k)$ of cylinder $m$-fans  of $\Delta_m(\mathbb R)^k$. 
 \end{enumerate}
\end{definition}

With a proof similar to Theorem \ref{thm:fond} we can prove: 

\begin{theorem}\label{thm:mfan}  A topological space $X$ is homeomorphic to $Spec(Free(m,k))$ for some $k$ if and only if it is $m$-fan  based.
\end{theorem} 

\begin{corollary} A topological space is the spectrum of a $\V(K_m)$-algebra if and only if it is a closed subset of an $m$-fan based space.
\end{corollary} 

Also Section \ref{sect:equiv} can be generalized and we have a functor from presented $\V(K_m)$-algebras to a category $\V(K_m)$-SIC, whose objects are triples $(X,C,m)$ where $X$ is an $m$-fan based space, $C$ is a spectral space and $m:C\to X$ is a monomorphic spectral map preserving closed sets. 

We now give an equivalent form of Theorem \ref{theor} valid whenever instead of $C$ we consider the Komori algebra $K_m$. 

\begin{definition} The rank of an MV-algebra $A$ is the cardinality of $A/Rad(A)$,  that is, the number of
radical classes.

 Let us call rank of a prime ideal $I\in Spec(A)$ the rank of the chain $A/I$. 

\end{definition}

\begin{lemma} A prime ideal $P$ of an MV-algebra $A$ has rank $k$ if and only if the maximal ideal $M_P$ containing $P$ has rank $k$. 
\end{lemma} 

\begin{proof} Let $P$ be a prime ideal of $A$. Suppose $P$ has rank $k$. Then $A/P$ has rank $k$ and is totally ordered. Since $P\subseteq M_P$, we have a surjective homomorphism $A/P\to A/M_P$. Since $P$ has rank $k$ and $M_P$ is maximal, $A/M_P$ is simple and has rank $k$.

 Conversely, suppose $A/M_P$ has rank $k$. Since there is a surjective homomorphism $A/P\to A/M_P$, $A/P$ must have also rank $k$. 
\end{proof} 

\begin{theorem}\label{theori} Let $A=Free(k)/I$ be an MV-algebra and $m\in\mathbb N$. The following  items are equivalent:

\begin{itemize} 
\item $A\in \V(K_m)$. 

\item Every ideal $P\in Spec(A)$ has rank divisor of $m$.

 \item Every ideal  $M\in Max(A)$ has rank divisor of $m$. 

 \item Every element of $V(I)$ has rank divisor of $m$. \end{itemize} \end{theorem} 
\begin{proof}
 Let us first prove that the first two points are equivalent. 

If $A$ is in $\V(K_m)$, then all its quotients modulo prime ideals are in $\V(K_m)$ and are totally ordered. Notice  that every totally  ordered MV-algebra $B$ in
 $\V(K_m)$ has rank dividing $m$. In fact, $B/Rad(B)$ is simple and is in $\V(K_m)$, so it has $k$ elements with $k|m$.

Conversely, suppose every prime $P$ has rank divisor of $m$. Then $A/P$ is in $\V(K_m)$. Now $A$ is a subdirect product of elements of $\V(K_m)$ so $A$ belongs to $\V(K_m)$ itself. 

The second and third points are equivalent by the previous lemma. 

The third and fourth points are also equivalent. In fact, an element of $Spec(A)$ has rank dividing $n$ if and only if the corresponding maximal in $Spec(A)$ 
has rank dividing $m$,  if and only if the corresponding maximal 
in $V(I)$ has rank dividing $m$. \end{proof}

\section{Main results}\label{group} 

This section is devoted to prime  spectra of $\ell$-groups:  they are the generalised spectral spaces which are completely normal, i.e., if $x$ and $y$ are in
the closure of a point $z$, then either $x$ is in the closure of $y$ or $y$ is in the closure of $x$.
Wehrung proved that the above properties characterise the second-countable
spectra of  $\ell$-groups and  there cannot be any first order
axiomatisation of the distributive lattices that are dual to spectra of $\ell$-groups.
Moreover, observe that from spectra of $\ell$-groups we cannot recover  $\ell$-groups, e.g. $\Z$ and $\R$ have the same prime spectrum. Other contribution to the prime spectra of $\ell$-groups is the paper by  Panti  \cite{P1} where  the prime ideals of finitely generated free $\ell$-groups are
characterised as the sets of piecewise homogeneous  linear functions with integer coefficients from $\R^n$ to $\R$ 
that vanish on a cone determined by a tuple of vectors.

 Here, using  the equivalence $\Delta$ between the category of $\ell$-groups and the category of perfect MV-algebras from \cite{DL1} we can characterize prime spectra of $\ell$-groups as follows:

\begin{theorem}  A topological space $X$ is the spectrum of an $\ell$-group $G$ if and only if there is a spectrum $Y$ of a perfect MV-algebra $A$ such that 
$X=Y\setminus \{M\}$, where $M$ is the only closed point in $Y$. 
\end{theorem} 

\begin{proof} Given $G$, the equivalence gives us $A=\Gamma(\mathbb Z\ lex\ G,(1,0))$. There is a homeomorphism $f$ between (prime) $\ell$-ideals of $G$ and (prime) proper ideals of $A$, given by $f(L)=\{(0,g)|g\in L\}$. 

Conversely, given a perfect MV-algebra $A$, the equivalence gives the $\ell$-group  $G=\{a-b|a,b\in Rad(A)\}$ and there is a homeomorphism $g$ between proper (prime) ideals of $A$ and prime ideals of $G$ given by $g(P)=\{a-0|a\in P\}$. \end{proof} 

From Corollary \ref{cor:vcspec} it follows: 

\begin{corollary} A topological space is the spectrum of a perfect MV-algebra if and only if it is a closed subset of a cylinder fan based space with only one closed point.
\end{corollary} 

\begin{proof} This follows because an MV-algebra is perfect if and only if it is $\V(C)$ and local see \cite[Proposition 5(5)]{BDG}.
\end{proof}

Summing up we have the theorem announced in the introduction: 

\begin{theorem} (Same as Theorem \ref{thm:lspec})  
A topological space $X$ is the spectrum of an abelian $\ell$-group $G$
if and only if there is a closed subset $Y$ of a cylinder fan based space $T$, 
such that $X=Y\setminus \{M\}$, where $M$ is the only closed point in $Y$.
\end{theorem} 

\section{From spectra of free MV-algebras to logics and geometries}\label{geo}

By addressing the problem of characterising the spectra of algebras, it can be seen that certain geometries are closely related to certain logics.

We see that the zerosets of Mc Naughton functions interpreted on models with a finite number of variables turn out to be rational polyhedra of the hypercube $[0,1]^n$, and also that a certain \L ukasiewicz logic can be seen as the logic of the geometry of rational polyhedra. 

Moreover, the family of zerosets of a finitely generated free algebra forms a  distributive lattice that is isomorphic to the lattice of compact open sets of the spectrum of the free algebra. In this way, via the aforementioned isomorphism, a relationship between the spectrum of a finitely generated free algebra, the geometry of rational polyhedra, which in fact is an affine geometry, and \L ukasiewicz logic becomes apparent.

What we have obtained above can be extended to the case of free MV-algebras with an infinite number of variables, up to replacing rational polyhedra  by rational cylinders of unit hypercube of infinite dimension. This was proved in the paper \cite{DL}.

\subsection{From spectra of free MV-algebras of the variety $\V(C)$ to logic}

Likewise to the MV-case, in the variety $\V(C)$, we can realize free $\V(C)$-algebras as algebras of functions from $\Delta(\mathbb R)^n$ to $\Delta(\mathbb R)$, and the $\V(C)$-polyhedra will be the zerosets of these functions, which are subsets of $\Delta(\mathbb R)^n$. We can ask how do these sets look like. 

We can identify $\Delta(\mathbb R)^n$ with the set of tuples $(v_1,x_1),\ldots,(v_n,x_n)$ where $v_1,\ldots,v_n\in\{0,1\}$ and $x_1,\ldots,x_n$ belong to the positive quadrant of $\mathbb R^n$. 

By \cite[Theorem 3.1]{DLV}  every polynomial function from $\Delta(\mathbb R)^n$ to $\Delta(\mathbb R)$ has the form 

$$P((v_1,x_1),\ldots,(v_n,x_n))=(f(v_1,\ldots,v_n),f_{(v_1,\ldots,v_n)}(x_1,\ldots,x_n))$$ 

where $v_1,\ldots,v_n\in\{0,1\}$, $f:\{0,1\}^n\to\{0,1\}$  and each $f_{(v_1,\ldots,v_n)}(x_1,\ldots,x_n)$ is an $\ell$-polynomial defined in the positive quadrant of $\mathbb R^n$. 

So the zeroset of $P$ is the set 

$$Z(P)=\{((v_1,x_1),\ldots,(v_n,x_n))| f(v_1,\ldots,v_n)=0\ AND\ f_{(v_1,\ldots,v_n)}(x_1,\ldots,x_n)=0\}.$$

So $Z(P)$ is a disjoint union of zerosets of $\ell$-polynomials, one for each tuple $v_1,\ldots,v_n$ such that $f(v_1,\ldots,v_n)=0$.

Now it is known that zerosets of $\ell$-polynomials in the positive quadrant of $\R^n$ coincide with rational cones contained in that quadrant.

So we can write

$$(1) \ Z(P)=\{((v_1,x_1),\ldots,(v_n,x_n))| (v_1,\ldots,v_n)\in S\ AND\ (x_1,\ldots,x_n)\in C_{(v_1,\ldots,v_n)}\},$$

where $S$ is a subset of $\{0,1\}^n$ and $C_{(v_1,\ldots,v_n)}$ is a cone depending on $(v_1,\ldots,v_n)$.

Conversely, consider a set Z of the form (1). So $S\subseteq \{0,1\}^n$, and for every $v\in S$, $C_v$ is a cone. Then $C_v$ is the zeroset of an $\ell$-polynomial  $f_v$. Then there is an MV-polynomial $P:[0,1]^n\to [0,1]$ (not unique) such that $P=(P(v),f_v)$ in a neighborhood of $v$ in $[0,1]^n$ if $v\in S$ and $P=1$ in a neighborhood of $v$ in $[0,1]^n$ if $v\notin S$. Then $Z$ is the zeroset of $P$ in $\Delta(\mathbb R)^n$. 

Summing up we have:

\begin{theorem} Zerosets of polynomials in $\Delta(\mathbb R)^n$ coincide with sets of the form (1). 
\end{theorem} 

\section{Homogeneity in MV-algebras}\label{sec:hom} 

In \cite{BDLE} the authors began a study of algebraic geometry  in MV-algebras inspired by the universal algebraic geometry of \cite{P2002}. In this section we are inspired by an issue of classical algebraic geometry over fields, that is, homogeneous polynomials and projective varieties. 

In \cite{DL04} it is shown that every McNaughton function is equivalent to a normal form $\phi(x)=\wedge\vee\rho(ax+b)$, where $\wedge$ denotes a finite conjunction, $\vee$ denotes a finite disjunction, $\rho(y)=0\vee(y\wedge 1)$ is the truncation function, $a$ is a vector of integers, $b$ is  an integer, and $ax+b$ is a  polynomial of degree one, possibly in $n$ variables (or a constant). 

It would be interesting to investigate in what sense the normal form of a McNaughton function is unique. For instance, there could be repetitions, and $0$ and $0\vee 0$ are two normal forms of the zero function. 

We can say that the function $\phi(x)=\wedge\vee\rho(ax+b)$ in normal form is homogeneous when every constant term $b$ is zero. So $\phi(x)=\wedge\vee\rho(ax)$.

We would like to pursue as far as possible the analogy between homogeneous McNaughton functions and usual homogeneous polynomials of algebraic geometry. 

For instance,  for every McNaughton function $\phi(x)=\wedge\vee\rho(ax+b)$ in $n$ variables there is a homogeneous (in our sense) McNaughton function $\psi(x,y)$ in $n+1$ variables such that $\phi(x)=\psi(x,1)$. That is, $\psi(x,y)=\wedge\vee\rho(ax+by)$. This is similar to the passage from affine varieties to projective varieties in usual algebraic geometry, where one adds a variable to make all monomials have the same degree. 

Moreover, given a zeroset $Z$ of a polynomial $\phi(x)=\wedge\vee\rho(ax+b)$, can we say that the zeroset of the polynomial $\phi(x)=\wedge\vee\rho(ax)$ is the set of  points at infinity of $Z$?  This is analogous to consider a variety, that is a zeroset of a set of polynomials over a field, and cancel all monomials whose degree is not the maximal one: the zeros of the remaining homogeneous polynomials in a sense are the  points at infinity of the variety. 

Let $\phi(x)$ be a McNaughton function in normal form, that is, $\phi(x)=\wedge\vee\rho(ax+b)$. 
If $\phi(x)$ is homogeneous then $\phi(x)=\wedge\vee\rho(ax)$, so $\phi(0)=0$. The converse is false:

\begin{theorem} There is a McNaughton function $\phi(x)$ such that $\phi(0)=0$ but $\phi$ is not homogeneous. 
\end{theorem} 

\begin{proof}

We begin with a lemma:

\begin{lemma} For every real number $b$, we have $\rho(b)=0$ if and only if $b\leq 0$.
\end{lemma} 

\begin{proof} $\rho(b)=0$ if and only if $max(0,min(1,b))=0$ if and only if $min(1,b)\leq 0$ if and only if $b\leq 0$. 
\end{proof} 

Now we have a criterion for the condition $\phi(0)=0$:

\begin{lemma} Let $\phi(x)$ be a McNaughton function in normal form, that is, $\phi(x)=\wedge\vee\rho(ax+b)$. Then $\phi(0)=0$ if and only if 
there is a disjunction $D$ where $b \leq 0$ for every $b\in D$.
\end{lemma}

\begin{proof} Suppose $\phi(0)=0$. Then $\phi(0)=\wedge\vee\rho(b)=0$. Since $\rho(b)\geq 0$ for every $b$, we must have $\vee\rho(b)=0$ for some disjunction D in 
$\phi$, so $\rho(b)=0$ for every $b$ in $D$, so $b\leq 0$ for every $b$ in $D$. 

Conversely, if there is a disjunction $D$ where $b\leq 0$ for every $b$ in $D$, then there is a disjunction $D$ where $\rho(b)=0$ for every $b\in D$, so 
$\vee\rho(b)=0$ and $\wedge\vee\rho(b)=0$, that is $\phi(0)=0$.\end{proof} 

In one variable we have:

\begin{lemma} Every homogeneous McNaughton function in one variable is either identically zero, or equal to $ax$ locally in zero for some positive integer $a$.
\end{lemma}

\begin{proof} Every homogeneous simple function $\rho(ax)$ has this property for every integer $a$. In fact, note that $\rho(ax)=ax$ locally in zero for every $a>0$, and $\rho(ax)=0$ when $a\leq 0$. These properties are preserved under finite join and meet, so they extend to every homogeneous McNaughton function in one variable. 
\end{proof} 

Now to conclude the proof of the theorem,  take for instance $\phi(x)=x\wedge \rho(2x-1)$. Clearly $\phi(0)=0$ and $\phi(x)$ is locally zero in $0$, but is not identically zero. So $\phi$ is not homogeneous. 
\end{proof}

It would be  interesting to characterize semantically homogeneous McNaughton functions in $n$ variables among the other McNaughton functions. 
Here we  classify zerosets of homogeneous McNaughton functions  as follows. Recall from \cite{M1} that rational polyhedra in $[0,1]^n$  coincide with zerosets of McNaughton functions (or of finite systems of McNaughton functions). 
For the homogeneous case first we notice:

\begin{lemma} A set $P$ is the zeroset of a finite set of homogeneous functions if and only if it is the zeroset of a single homogeneous function.  
\end{lemma} 

\begin{proof} This is because if $\phi_1,\ldots,\phi_n$ are homogeneous McNaughton functions, then $\phi_1\vee\ldots\vee\phi_n$ is equivalent to a homogeneous McNaughton function. 
\end{proof} 

Recall that, by definition, a rational cone passing through the origin in $[0,1]^n$ is a finite intersection of halfspaces $ax\leq 0$, where $a$ is a vector of integers. 

Now let us continue the investigation:

\begin{lemma} A rational polyhedron $P\subseteq [0,1]^n$ is the zeroset of a finite set of homogeneous McNaughton functions if and only if $P$ is a finite union of rational cones passing through the origin.
\end{lemma}

\begin{proof} If $P$ is a zeroset as above, then $P$ is a zeroset of a function of the form $\wedge\vee\rho(ax)$, so $P$ is
a finite union of finite intersections of sets of the form $\rho(ax)=0$, or equivalently of half spaces $ax\leq 0$. Now a finite intersection of such half spaces is a rational cone passing through the origin.

Conversely, every rational cone passing through the origin is a finite intersections of sets of the form $\rho(ax)=0$, so every finite union of rational cones passing through the origin is the zeroset of a homogeneous function $\wedge\vee\rho(ax)$. 
\end{proof} 

As usual, we say that a subset $C\subseteq [0,1]^n$ is a cone if whenever $x\in C$ and $\lambda \geq 0$ and  $\lambda x\in[0,1]^n$, then $\lambda x\in C$. 

We can say more on cones:

\begin{corollary} A rational polyhedron $P\subseteq [0,1]^n$ is the zeroset of a finite set of homogeneous McNaughton functions if and only if $P$ is a cone.
\end{corollary} 

\begin{proof} Clearly every finite union of rational cones is a cone. 

Conversely, suppose $P$ is a rational polyhedron and is a cone. $P$ is a finite union of rational simplexes. However, the cone generated by a rational simplex is a rational cone. So $P$ is a finite union of rational cones and is the zeroset of a homogeneous McNaughton function. 
\end{proof} 

We note that all results extends to finite systems in infinitely many variables:

\begin{corollary} Let $I$ be an infinite set. A cylinder rational polyhedron $P\subseteq [0,1]^I$ is the zeroset of a finite set of homogeneous McNaughton functions if and only if $P$ is a cylinder cone.
\end{corollary} 

\subsection{Locally homogeneous McNaughton functions}\label{locallyhom}

Clearly homogeneity of polynomials makes no sense in MV-algebras because there is no notion of monomial and degree. 

The next definition gives a property of McNaughton functions analogous to homogeneity for usual polynomials.

We say that a McNaughton function $\phi$ is locally homogeneous in zero if there is a neighborhood $U$ of $0$ such that, if $n\in \N$ and $x$ and $nx$ belong to $U$, then $\phi(nx)=n\phi(x)$, where the right hand side denotes  real multiplication in the usual sense.  

\begin{lemma} If $\phi(x)=\wedge\vee\rho(ax)$ then $\phi$ is locally homogeneous in zero. \end{lemma} 

\begin{proof} By induction on $\phi$. If $\phi(x)=\rho(ax)$ then in a neighborhood $U$ of $0$ we have $\phi(x)=max(0,ax)$, but $max(0,nax)=nmax(0,ax)$, so if $nax\in U$ then $\phi(nax)=n\phi(ax)$. \end{proof} 

\begin{proposition} A function $\phi(x)=\wedge\vee\rho(ax+b)$ is locally homogeneous in zero if and only if there is a disjunction $D$ of $\phi$ such that  $b\leq 0$ in every disjunct of $D$. \end{proposition} 

\begin{proof} Suppose  there is a disjunction where $b\leq 0$ always. Then  locally $\phi$ is a lattice combination of simple locally homogeneous functions (because the non locally homogeneous simple functions locally cancel out) and is locally homogeneous. 

Otherwise, every disjunction $D$ contains a simple function $f_D$ with $b>0$, and simple functions $\rho(ax+b)$ with $b>0$ are locally $(ax+b)\wedge 1$ which is locally greater than $1/2$ in zero. Now $\phi$ is greater than the disjunction of the $f_D$, so $\phi$ is locally greater than $1/2$ in zero and cannot be locally homogeneous in zero. 
\end{proof} 

\begin{corollary} A function $\phi(x)=\wedge\vee\rho(ax+b)$ is locally homogeneous in zero if and only if $\phi(0)=0$.
\end{corollary} 

\begin{proof} If $\phi$ is locally homogeneous then clearly $\phi(0)=0$, otherwise it has a positive limit in $0$ and $\phi$ cannot be homogeneous in any neighborhood of $0$. Conversely, suppose $\phi(0)=0$. Then there must be a disjunction $D$ of $\phi$ such that  $b\leq 0$ in every disjunct of $D$, otherwise, as in the proof of the previous proposition, $\phi\geq 1/2$ locally in $0$, contrary to the hypothesis $\phi(0)=0$. So by the previous proposition,  $\phi$ is locally homogeneous in $0$.  
\end{proof}

\section{The mod/theor connection}\label{mod}

Mundici in \cite{M1} establishes a so called mod/theor connection. On the one hand,  there is {\L}ukasiewicz logic, which is a theory, that is a set of formulas and a syntactic consequence relation $\Phi\vdash\phi$ between sets of formulas and formulas. Usually theories are defined as deductively closed sets of formulas, that is $T\vdash\phi$ implies $\phi\in T$. 

 On the other hand, there is a set of so-called models and a semantic consequence relation $m\models\phi$, where $m$ is a model and $\phi$ is a formula. Then we 
can call $Mod(\phi)$ the set of all models of the formula $\phi$. 

We note that the choice of the models and of the semantic relation is not unique. For instance, in \cite{M1}, the considered models are elements of $[0,1]^n$, and formulas are in the variables $x_1,\ldots,x_n$. $[0,1]$ indeed it generates the variety of MV-algebras, which are the algebras of {\L}ukasiewicz logic. 

Then it turns out that models of formulas coincide with rational polyhedra, and in this sense one can say {\L}ukasiewicz logic is the logic of rational polyhedra. We think it would be of interest to express more formally this idea. 

However, in our approach the possibility of considering infinitely many variables is important, so for us, models are elements of $[0,1]^I$, where $I$ is any set, and formulas are in the variables $x_i$, with $i\in I$. It turns out that models of formulas  coincide with cylinder rational polyhedra. So, in some sense we are justified in saying that {\L}ukasiewicz logic is the logic of cylinder rational polyhedra as well.

Note that the mod/theor connection restricts to the Boolean case, both in the finite and infinite variable variant, and classical logic. Then models are elements of $\{0,1\}^n$, and $\{0,1\}$ generates the variety of Boolean algebras, which are the algebras of classical logic.

 It turns out that models of formulas in finitely many variables are all the subsets of $\{0.1\}^n$. In infinitely many variables, the cylindrification idea still works, and we could talk of cylinder-finite subsets of $\{0,1\}^I$ as set $\{f\in[0,1]^I|f|_Y\in F\}$, where $Y$ is a finite subset of $I$ and $F$ is any subset of $[0,1]^Y$. 

These ideas can be extended in many ways. For instance, we could consider axiomatic extensions of {\L}ukasiewicz logic. Classical logic is essentially 
{\L}ukasiewicz logic plus the axiom $x\oplus x=x$. Here, we focus on Chang variety $\V(C)$ of MV-algebras corresponds to {\L}ukasiewicz logic plus the axiom $2(x^2)=(2x)^2$. In this case one chooses $m\in\Delta(\R)^n$, since $\Delta(\R)$ generates Chang variety, like $\{0,1\}$ generates the variety of Boolean algebras and $[0,1]$ generates the variety of MV-algebras. In this case, we can say that the logic of $V(C)$-variety is the logic of rational fans, in the case of finite models, and of Cylinder rational fans, in the infinite case. Of course a more extensive analysis can be done by considering all the equational extensions of {\L}ukasiewicz logic.

\section{Conclusions}\label{conclusio}

As we have seen, the main goals of this paper are two: the construction of new functors, and the study of spectra of subclasses of MV-algebras. 

We have seen in this paper that indeed, using classical dualities, it is possible to build new functors involving the category of MV-algebras, which may give some information on this category. Moreover, we think that characterizing the spectra of subclasses of MV-algebras is an interesting research programme.  As it is shown by the results of this paper, we can take advantage that, by Stone duality, two MV-algebras have homeomorphic spectrum if and only if they have isomorphic lattices of principal ideals, so that the problem can be recast in lattice theoretic form. 

Moreover, following  Mundici who  states that {\L}ukasiewicz logic is the logic of rational polyhedra, we continue in building a bridge between polyhedral geometry and MV-algebras (and $\ell$-groups) theory. So from the previous results we can draw the following conclusions:

1) rational polyhedra of $[0,1]^n$ are the algebraic varieties of 	\L ukasiewicz logic;

2) rational cones of $[0,1]^n$ are in bijection with rational cones of $\R^n$ and the latter are the algebraic varieties of the equational logic of $\ell$-groups;

3) cylinder polyhedra of $[0,1]^I$, with $I$ a set possibly infinite, are the algebraic varieties of \L ukasiewicz logic in infinitely many variables, that is, zerosets of single polynomials in $[0,1]^I$;

4) likewise cylinder fans are the algebraic varieties of the same logic in item  2) but over infinitely many variables;

5) fans in $\Delta(\R)^n$ are the algebraic varieties of \L ukasiewicz logic plus the MV-algebraic axiom $(2x)^2=2(x^2)$, which axiomatizes the variety $\V(C)$.  

\medskip

This paper can be extended in many directions.

As we said, Plotkin in \cite{P2002} shows how to perform algebraic geometry in universal algebra. For the variety of MV-algebras, the theory has been conducted in \cite{BDLE}, and for $\ell$-groups, in \cite{DLV18}. One may continue by  considering subvarieties of MV-algebras, in particular the varieties generated by a Komori chain (where Chang's MV-algebra $C$ is the most simple example).

\end{document}